\documentclass{amsart}

\usepackage{graphicx}%
\usepackage{multirow}%
\usepackage{amsmath,amssymb,amsfonts}%
\usepackage{amsthm}%
\usepackage{mathrsfs}%
\usepackage[title]{appendix}%
\usepackage{xcolor}%
\usepackage{textcomp}%
\usepackage{manyfoot}%
\usepackage{algorithm}%
\usepackage{algorithmicx}%
\usepackage{algpseudocode}%
\usepackage{listings}%
\theoremstyle{thmstyleone}
\usepackage{dsfont,mathrsfs}

\usepackage[
a4paper,
inner=3.0cm,
outer=3.0cm,
top=3.0cm,
bottom=3.0cm
]{geometry}

\DeclareMathOperator*{\esssup}{ess\,sup}
\DeclareMathOperator*{\essinf}{ess\,inf}

\usepackage[breaklinks]{hyperref}

\newtheorem{theorem}{Theorem}[section]
\newtheorem{lemma}[theorem]{Lemma}

\theoremstyle{definition}
\newtheorem{definition}[theorem]{Definition}

\newtheorem{corollary}[theorem]{Corollary}
\newtheorem{proposition}[theorem]{Proposition}

\theoremstyle{remark}
\newtheorem{remark}[theorem]{Remark}

\numberwithin{equation}{section}

\begin{document}

\title[DRBSDEs with Inhomogeneous simple Lévy processes \& Dynkin games]{Doubly reflected BSDEs driven by Inhomogeneous simple Lévy processes: Applications to generalized Dynkin games}


\author{Badr Elmansouri}
\address{Cadi Ayyad University (UCA), National School of Applied Sciences of Marrakech (ENSA-M)}
\curraddr{BP 575, Avenue Abdelkrim Khattabi, Guéliz, Marrakech, 40000, Morocco}
\email{b.elmansouri@uca.ac.ma}
\thanks{The first author, \textbf{Badr Elmansouri}, wishes to express his sincere gratitude to Professors \textbf{Mohamed El Otmani} and \textbf{Mohamed El Jamali} for their valuable contributions to the study of BSDEs driven by inhomogeneous simple Lévy processes and their financial applications, as well as for the insightful discussions that helped improve this work.}

\author{Ibtissam Hdhiri}
\address{Department of Mathematics, Faculty of Sciences of Gabès, University of Gabès}
\curraddr{LR17ES11, Gabès, 6072, Tunisia}
\email{ibti.hdhiri@gmail.com}

\subjclass[2010]{Primary: 60H05, 60H15, 60H20. Secondary: 60H30.}

\keywords{Backward stochastic differential equations, Inhomogeneous Lévy process, American game option, Dynkin game}

\date{\today}


\begin{abstract}
	We study doubly reflected backward stochastic differential equations with jumps and two completely separated right-continuous with left limits barriers in a filtration generated by an inhomogeneous Lévy processes. We establish existence and uniqueness results under a stochastic Lipschitz condition on the driver by means of a penalization method. We also prove a comparison principle and present two closely related applications. The first concerns the nonlinear valuation of an American game option in such a Lévy market, while the second addresses the associated generalized Dynkin game under nonlinear expectation. Moreover, under suitable semicontinuity assumptions on the barriers, we establish the existence of a saddle point for the game.
\end{abstract}

\maketitle

\section{Introduction}
The notion of doubly reflected backward stochastic differential equations (DRBSDEs, for short) has proven to be a powerful tool for addressing various financial problems, particularly those related to the pricing and hedging of American game options with a deterministic time horizon $T \in (0,+\infty)$ in a variety of market models, whether perfect or imperfect. This concept was first introduced by Cvitanic and Karatzas \cite{Cvitanic1996} in a Brownian framework consisting of a complete probability space $(\Omega, \mathcal{F}, \mathbb{P})$ and the natural filtration $\mathbb{F} := (\mathcal{F}_t)_{t \leq T}$ generated by a standard Brownian motion $B$. In this setting, the authors studied the existence and uniqueness of solutions to the following DRBSDE:
\begin{equation}
	\left\{
	\begin{split}
		\text{(i)} &~ \mathbb{P}\text{-a.s. for all } t \in [0,T]\\
		&~ Y_{t}=\xi+\int_{t}^{T} f(s,Y_s,Z_s,V_s)ds+\left( K^{+}_{ T}-K^{+}_{t}\right)-\left( K^{-}_{ T}-K^{-}_{t }\right)
		-\int_{t}^{T} Z_s d B_s;\\
		\text{(ii)} &~  L_t \leq Y_t \leq U_t,~  \forall t \in [0,T]~\text{a.s.};\\
		\text{(iii)} &~ \text{Skorokhod condition:}
		\int_0^{T  }(Y_{t}-L_{t})dK^{+}_t=\int_0^{T  }(U_{t}-Y_{t})dK^{-}_t=0 \text{ a.s.}
	\end{split}
	\right.
	\label{basic intro}
\end{equation}
The processes denoted by $K^+$ and $K^-$ are two continuous, increasing reflecting processes that intervene when necessary to keep the state process $Y$ within the bounds imposed by the lower barrier $L$ and the upper barrier $U$ (also referred to as obstacles). In \cite{Cvitanic1996}, the authors established the existence and uniqueness of a solution under the assumption that the coefficient $f$ is Lipschitz continuous and that the barriers satisfy Mokobodzki's condition. This condition requires the existence of two non-negative supermartingales lying between the lower and upper barriers. Furthermore, it was shown that in a specific case, the equation admits a unique solution via the penalization method, provided that one of the reflecting barriers can be uniformly approximated by Itô processes. An application to the valuation of a classical Dynkin game was also presented. 

Following this line of research, several authors sought to relax the assumptions imposed on the barriers $L$ and $U$, as those adopted in \cite{Cvitanic1996} were considered overly restrictive. In this context, Hamadène and Hassani \cite{Hamadène2005} introduced the notion of local solutions under the complete separation condition
$$
L_t<U_t,\qquad \text{for all } t\in[0,T],
$$
and for a Lipschitz continuous driver. Building on this approach, Hamadène \cite{Hamadène2006} applied the notion of local solutions to the pricing of American game options in the standard Black--Scholes market; see also \cite{Essaky2016} and \cite{Hamadène2000}.

Under the same separation condition on the barriers, Hamadène and Hdhiri \cite{HamadeneHdhiri2006} studied DRBSDEs with a continuous generator satisfying a quadratic growth condition. They also established a connection with mixed zero-sum stochastic differential games and provided a financial application concerning the characterization of the yield of a recallable option under Knightian uncertainty and an exponential utility function.

Extending the analysis beyond the Brownian framework to a setting with jumps driven by a Poisson random measure, Hamadène and Wang \cite{HamadeneWang2009} studied DRBSDEs of the form \eqref{basic intro} in a Brownian--Poisson filtration. They established existence and uniqueness results for DRBSDEs with right-continuous with left limits (RCLL) obstacles by again relying on local solutions under the complete separation conditions
$$
L_t<U_t,\qquad \text{for all }t\in[0,T),
$$
and
$$
L_{t-}<U_{t-},\qquad \text{for all }t\in(0,T],
$$
together with a Lipschitz continuous driver $f$. In addition, they established a connection between the DRBSDE and a mixed zero-sum differential--integral game; see also \cite{HamadeneHassaniOuknine2010} and \cite{Lepeltier2007} for related results.

Another line of research focuses on weakening the Lipschitz condition on the driver $f$, replacing it with a stochastic Lipschitz condition. The motivation behind this consideration lies in the study of more realistic extensions of the Black–Scholes model, where model parameters are not necessarily bounded or deterministic. These include the risk-free rate, the stock volatility process, and other stochastic factors. In such cases, the pricing of American game options naturally leads to imposing a stochastic Lipschitz condition on the generator. This approach was first developed in the seminal work of El Karoui and Huang \cite{ElKaroui1997} for classical BSDEs in a general probability space (see also \cite{Elmansouri2024VMSTA} for a broader framework). It was later studied in the context of DRBSDEs in the Brownian setting by Marzougue and El Otmani \cite{MarzougueOtmani2017}, and extended by the same authors in \cite{MarzougueOtmani2021} to the Brownian–Poisson framework under the complete separation of barriers. Their analysis was motivated by the pricing problem of an American game option in a discontinuous market model with unbounded stochastic parameters. This problem was further addressed by El Otmani et al.~\cite{OtmaniJamaliMarzougue2022} in a Lévy market model with a stochastic interest rate for the risk-free asset. More recently, Elmansouri and El Otmani \cite{ElmansouriOtmani2024} extended these results to a general filtration setting driven by a right-continuous with left limits (RCLL) martingale. The authors established existence and uniqueness of solutions and provided a pricing characterization for a game option between two insiders with asymmetric information in an Azéma-type market (see also \cite{ElJamali2024,ElmansouriElOtmani2025} for a related contribution).

Recently, El Jamali and El Otmani \cite{JamaliOtmani2019} considered the class of inhomogeneous Lévy processes introduced by Kluge \cite{kluge2005time}, motivated by their various financial applications. In \cite{JamaliOtmani2019}, the authors established the chaotic representation property of such processes and applied their results to the pricing of European options in a non-homogeneous Lévy market. It is worth noting that a time-inhomogeneous (or non-homogeneous) Lévy process—also referred to as a process with independent increments and absolutely continuous characteristics—generalizes the notion of a standard (homogeneous) Lévy process. Unlike their homogeneous counterparts, non-homogeneous Lévy processes are characterized by non-stationary increments. Building on this, the same authors addressed in \cite{ElJamaliOtmani2022} the pricing of American game options via reflected BSDEs in a discontinuous market model, where the jumps arise from a non-homogeneous Poisson random measure. It should be emphasized that all of the above pricing problems were addressed using the classical linear expectation framework, which naturally leads to the standard formulation of Dynkin games in perfect markets.

Motivated by these developments, the present paper investigates doubly reflected BSDEs driven by both a Brownian motion and an independent martingale measure associated with an inhomogeneous Poisson random measure, thereby extending the existing literature on DRBSDEs to more general jump frameworks. In our setting, the reflecting barriers are assumed to be completely separated, and we replace the classical Mokobodzki condition with a more general condition allowing jumps. The driver is assumed to satisfy a stochastic Lipschitz condition. Under suitable square-integrability assumptions on the data, we establish existence and uniqueness results using a priori estimates and a penalization approximation method. Subsequently, we prove a general comparison theorem that extends the result of \cite{JamaliOtmani2019}, allowing for differences in the drivers, terminal conditions, and reflecting barriers. This result enables us to address the valuation of an American game option in a non-homogeneous Lévy market, where the game values are expressed using a nonlinear expectation rather than a classical one. As a result, we are led to study a generalized Dynkin game in a financial context. Additionally, under further regularity assumptions on the obstacles, we establish the existence of a saddle-point under the nonlinear expectation operator.

The paper is organized as follows: In Section \ref{sec1}, we introduce the mathematical framework. Section \ref{DRBSDE} formulates the DRBSDE, states the assumptions on the data, and provides key a priori estimates, which are then used to prove the existence and uniqueness of a solution. Section \ref{sec3} is devoted to establishing a general comparison principle for such DRBSDEs. Finally, Section \ref{sec44} is devoted to two closely related applications of our results. First, we study the nonlinear pricing of American game options in a possibly imperfect market driven by an inhomogeneous Lévy process. We then formulate the associated generalized Dynkin game under nonlinear expectation and establish the existence of a saddle point under suitable semicontinuity assumptions on the reflecting barriers.

\section{Preliminaries and assumptions}\label{sec1}
Let $T>0$ be a fixed time corresponding to the time horizon of all the models considered in what follows. We consider a complete probability space $(\Omega,\mathcal{F},\mathbb{P})$ carrying a standard one-dimensional Brownian motion $\left(B_t\right)_{t\leq T}$ and compensated Poisson martingales $\big(\widetilde{N}_t^{(k)}\big)_{t\leq T}$ associated with independent inhomogeneous Poisson processes
\[
N^{(k)}:=\big(N_t^{(k)}\big)_{t\leq T},
\qquad k\in\{1,\ldots,d\},
\]
where $d\in\mathbb{N}^{\ast}$. For every $t\in[0,T]$, we have
\[
\widetilde{N}_t^{(k)}
:=
N_t^{(k)}-\int_0^t\lambda_s^{(k)}\,ds,
\]
where the intensity function
\[
[0,T]\ni t\longmapsto
\lambda_t
=
\big(\lambda_t^{(1)},\ldots,\lambda_t^{(d)}\big)
\]
is deterministic and positive. Moreover,
\[
d\big\langle \widetilde{N}^{(k)},
\widetilde{N}^{(k')}\big\rangle_t
=
\lambda_t^{(k)}\delta_{k,k'}\,dt,
\]
where $\big\langle \widetilde{N}^{(k)},
\widetilde{N}^{(k')}\big\rangle$ denotes the predictable quadratic covariation of $\widetilde{N}^{(k)}$ and $\widetilde{N}^{(k')}$, and $\delta_{k,k'}$ denotes the Kronecker delta.

On $(\Omega,\mathcal{F},\mathbb{P})$, we define the filtration
$\mathbb{F}:=(\mathcal{F}_t)_{t\leq T}$ by
\[
\mathcal{F}_t
:=
\bigcap_{\varepsilon>0}
\left(
\mathcal{F}^0_{(t+\varepsilon)\wedge T}
\vee\mathcal{N}
\right),
\qquad t\in[0,T],
\]
where $\mathcal{N}$ denotes the collection of all $\mathbb{P}$-null sets in
$\mathcal{F}$ and
\[
\mathcal{F}^0_t
:=
\sigma\left\{X_s:\,s\leq t\right\},
\qquad t\in[0,T].
\]
Using the Brownian motion and the independent inhomogeneous Poisson
processes introduced above, we define the process
$X=(X_t)_{t\leq T}$ by
\[
X_t
=
\int_0^t\sqrt{c_s}\,dB_s
+
\sum_{k=1}^d
\int_0^t\gamma_s^{(k)}\,d\widetilde{N}_s^{(k)},
\qquad t\in[0,T].
\]
The process $X$ is an inhomogeneous Lévy process, that is, an additive
process with independent increments, continuous in probability, and
satisfying $X_0=0$.

Here, the volatility coefficient
\[
[0,T]\ni t\longmapsto c_t\in\mathbb{R}_{+}^{\ast}
\]
and the deterministic coefficients
$\big(\gamma_t^{(k)}\big)_{t\leq T}$, $k\in\{1,\ldots,d\}$, satisfy
\[
\int_0^T c_s\,ds
+
\sum_{k=1}^d
\int_0^T
\big|\gamma_s^{(k)}\big|^2\lambda_s^{(k)}\,ds
<+\infty.
\]
It follows from the definition of the inhomogeneous Lévy process
$X=(X_t)_{t\leq T}$ that $X$ is an RCLL and locally square-integrable
martingale with respect to $\mathbb{F}$. Moreover,
\[
\mathbb{E}\left[[X,X]_T\right]
=
\int_0^T c_s\,ds
+
\sum_{k=1}^d
\int_0^T
\big|\gamma_s^{(k)}\big|^2\lambda_s^{(k)}\,ds
<+\infty,
\]
where $[X,X]$ denotes the quadratic variation of $X$. Recall that, for
any square-integrable martingale $(M_t)_{t\leq T}$,
$
\mathbb{E}\left[[M,M]_T\right]
=
\mathbb{E}\left[\langle M,M\rangle_T\right]
$ 
(see Proposition 4.50 in \cite[p.~53]{Jacod2013}). Therefore, by
Proposition 4.50(c) in \cite[p.~53]{Jacod2013}, we conclude that $X$ is
a square-integrable $\mathbb{F}$-martingale.

For further details on time-inhomogeneous Lévy processes and their
financial applications, we refer the reader to
\cite{kluge2005time}.

Let us denote by $\mathcal{X}_{t-} := \lim\limits_{s \nearrow t} \mathcal{X}_s$ the left limit at time $t \in (0,T]$ of a given RCLL process $\mathcal{X}$, with the convention $\mathcal{X}_{0-} = \mathcal{X}_{0}$. The jump size at time $t \in [0,T]$ is denoted by $\Delta \mathcal{X}_t := \mathcal{X}_t - \mathcal{X}_{t-}$. Let us also denote by $\mathcal{T}_{\tau_1}^{\tau_2}$ the set of $[0,T]$-valued $\mathbb{F}$-stopping times $\tau$ such that $\tau_1 \leq \tau \leq \tau_2$, a.s., where $\tau_1$ and $\tau_2$ are two $[0,T]$-valued $\mathbb{F}$-stopping times satisfying $\tau_1 \leq \tau_2$, a.s. Furthermore, we denote the conditional expectation with respect to the $\sigma$-algebra $\mathcal{F}_t$ by $\mathbb{E}^{\mathcal{F}_t}[\cdot]$. Finally, for $x \in \mathbb{R}$, we recall that $x^{+} = \max(x,0)$ and $x^{-} = \min(-x,0) = -\min(x,0)=\max(-x,0)$.

To describe the parameters and the solution of our equation, we introduce the following processes and spaces. Let $\tau\in\mathcal{T}_0^T$ be a stopping time, and let $(\alpha_t)_{t\leq T}$ be a non-negative $\mathbb{F}$-adapted process. We associate with $\alpha$ the continuous increasing process $(A_t)_{t\leq T}$ defined by
\[
A_t:=\int_0^t \alpha_s^2\,ds,
\qquad t\in[0,T].
\]
For a fixed constant $\beta>0$, we then define the following spaces.

\begin{itemize}
	\item $\mathbb{L}^2_{\beta}(\mathcal{F}_\tau)$: the space of all $\mathbb{R}$-valued, $\mathcal{F}_{\tau}$-measurable random variables $\xi$ such that
	\[
	\left\|\xi\right\|^2_{\mathbb{L}^2_{\beta}(\mathcal{F}_\tau)}
	:=
	\mathbb{E}\left[e^{\beta A_{\tau}}|\xi|^2\right]
	<+\infty.
	\]
	
	\item $\ell^2$: the space $\mathbb{R}^d$ endowed with the norm
	\[
	\|v\|_{\ell^2}^2
	:=
	\sum_{k=1}^d |v^{(k)}|^2,
	\qquad
	v=\big(v^{(1)},v^{(2)},\ldots,v^{(d)}\big)\in\mathbb{R}^d.
	\]
	
	\item $\mathcal{K}^2$: the space of all $\mathbb{R}$-valued, $\mathbb{F}$-predictable, RCLL, increasing processes $(K_t)_{t\leq T}$ such that $K_0=0$ and
	\[
	\|K\|_{\mathcal{K}^2}^2
	:=
	\mathbb{E}\left[|K_T|^2\right]
	<+\infty.
	\]
	
	\item $\mathcal{H}^2$: the space of all $\mathbb{R}$-valued, $\mathbb{F}$-predictable processes $(Z_t)_{t\leq T}$ such that
	\[
	\|Z\|_{\mathcal{H}^2}^2
	:=
	\mathbb{E}\left[\int_0^T |Z_s|^2\,ds\right]
	<+\infty.
	\]
	
	\item $\mathcal{V}^2_\lambda$: the space of all $\mathbb{R}^d$-valued, $\mathbb{F}$-predictable processes $(V_t)_{t\leq T}$ such that
	\[
	\begin{aligned}
		\|V\|_{\mathcal{V}^2_\lambda}^2
		&:=
		\mathbb{E}\left[
		\int_0^T
		\left\|V_s\circ\sqrt{\lambda_s}\right\|_{\ell^2}^2\,ds
		\right] \\
		&=
		\sum_{k=1}^d
		\mathbb{E}\left[
		\int_0^T
		|V_s^{(k)}|^2\lambda_s^{(k)}\,ds
		\right]
		<+\infty,
	\end{aligned}
	\]
	where $V_s\circ\sqrt{\lambda_s}$ denotes the Hadamard product of the vectors $V_s$ and $\sqrt{\lambda_s}$, namely,
	\[
	V_s\circ\sqrt{\lambda_s}
	:=
	\left(
	V_s^{(1)}\sqrt{\lambda_s^{(1)}},
	\ldots,
	V_s^{(d)}\sqrt{\lambda_s^{(d)}}
	\right).
	\]
	For convenience, we set
	\[
	\|V_s\|_\lambda^2
	:=
	\left\|V_s\circ\sqrt{\lambda_s}\right\|_{\ell^2}^2
	=
	\sum_{k=1}^d |V_s^{(k)}|^2\lambda_s^{(k)}.
	\]
	
	\item $\mathcal{S}^2_{\beta}$: the space of all $\mathbb{R}$-valued, $\mathbb{F}$-adapted, RCLL processes $(Y_t)_{t\leq T}$ such that
	\[
	\|Y\|_{\mathcal{S}^2_{\beta}}^2
	:=
	\mathbb{E}\left[
	\sup_{0\leq t\leq T}
	e^{\beta A_t}|Y_t|^2
	\right]
	<+\infty,
	\]
	with the convention $\mathcal{S}^2:=\mathcal{S}^2_0$.
	
	\item $\mathcal{S}^{2,\alpha}_{\beta}$: the space of all $\mathbb{R}$-valued, $\mathbb{F}$-adapted processes $(Y_t)_{t\leq T}$ such that
	\[
	\|Y\|_{\mathcal{S}^{2,\alpha}_{\beta}}^2
	:=
	\mathbb{E}\left[
	\int_0^T
	e^{\beta A_s}|\alpha_sY_s|^2\,ds
	\right]
	<+\infty.
	\]
	
	\item $\mathcal{H}^2_{\beta}$: the space of all $\mathbb{R}$-valued, $\mathbb{F}$-predictable processes $(Z_t)_{t\leq T}$ such that
	\[
	\|Z\|_{\mathcal{H}^2_{\beta}}^2
	:=
	\mathbb{E}\left[
	\int_0^T
	e^{\beta A_s}|Z_s|^2\,ds
	\right]
	<+\infty.
	\]
	
	\item $\mathcal{V}^2_{\lambda,\beta}$: the space of all $\mathbb{R}^d$-valued, $\mathbb{F}$-predictable processes $(V_t)_{t\leq T}$ such that
	\[
	\begin{aligned}
		\|V\|_{\mathcal{V}^2_{\lambda,\beta}}^2
		&:=
		\mathbb{E}\left[
		\int_0^T
		e^{\beta A_s}\|V_s\|_\lambda^2\,ds
		\right] \\
		&=
		\sum_{k=1}^d
		\mathbb{E}\left[
		\int_0^T
		e^{\beta A_s}|V_s^{(k)}|^2
		\lambda_s^{(k)}\,ds
		\right]
		<+\infty.
	\end{aligned}
	\]
	
	\item
	$
	\mathfrak{B}^2_{\beta}
	:=
	\left(
	\mathcal{S}^2_{\beta}
	\cap
	\mathcal{S}^{2,\alpha}_{\beta}
	\right)
	\times
	\mathcal{H}^2_{\beta}
	\times
	\mathcal{V}^2_{\lambda,\beta}.
	$
	
	\item
	$
	\mathfrak{D}^2_{\beta}
	:=
	\left(
	\mathcal{S}^2_{\beta}
	\cap
	\mathcal{S}^{2,\alpha}_{\beta}
	\right)
	\times
	\mathcal{H}^2_{\beta}
	\times
	\mathcal{V}^2_{\lambda,\beta}
	\times
	\mathcal{K}^2
	\times
	\mathcal{K}^2.
	$ 
\end{itemize}

Finally, let us recall the predictable representation property for a non-homogeneous Lévy process X that admits a finite number $d-1$ of jump sizes, that is verifies within the filtration $\mathbb{F}$, which states that any every square integrable $\mathcal{F}_T$-measurable random variable $\xi$ admits a unique decomposition of the following form:
$$
\xi=\mathbb{E}\left[\xi\right]+\int_{0}^{T} Z_s dB_s+\sum_{k=1}^{d} \int_{0}^{T} V^{(k)}_s d\widetilde{N}^{(k)}_s,\quad t \in [0,T],
$$
where $Z \in \mathcal{H}^2$ and $V \in \mathcal{V}^2_\lambda$ (see, e.g. \cite[Section 2.3]{JamaliOtmani2019} for more details).

\section{Doubly reflected BSDEs: Existence and uniqueness results}
\label{DRBSDE}
\subsection{Presentation}
In this first part, we are looking for a quintuplet of processes $(Y_t,Z_t,V_t,K^{+}_t,K^{-}_t)_{t \leq T}$ such that:
\begin{equation}
	\left\{
	\begin{split}
		\text{(i)} &~ \mathbb{P}\text{-a.s. for all } t \in [0,T]\\
		&~ Y_{t}=\xi+\int_{t}^{T} f(s,Y_s,Z_s,V_s)ds+\left( K^{+}_{ T}-K^{+}_{t}\right)-\left( K^{-}_{ T}-K^{-}_{t }\right)\\
		&\qquad\qquad-\int_{t}^{T} Z_s d B_s-\sum_{k=1}^d\int_{t }^{T}V^{(k)}_s d \widetilde{N}^{(k)}_s;\\
		\text{(ii)} &~  L_t \leq Y_t \leq U_t,~  \forall t \in [0,T]~\text{a.s.};\\
		\text{(iii)} &~ \text{Skorokhod condition:}
		\int_0^{T  }(Y_{t-}-L_{t-})dK^{+}_t=\int_0^{T  }(U_{t-}-Y_{t-})dK^{-}_t=0 \text{ a.s.}
	\end{split}
	\right.
	\label{basic equation}
\end{equation}
\begin{remark}
	\begin{itemize}
		\item[$\bullet$]
		The state process $Y$ in DRBSDE (\ref{basic equation}) exhibits two types of jumps. The first type consists of the totally inaccessible jumps determined by the purely discontinuous martingales part $\big(\widetilde{N}_t^{(k)}\big)_{t \leq T}$. The second type is the predictable jumps, which stem from the negative jumps of lower obstacle $L$ and the positive jumps of upper barrier $U$.
		\item[$\bullet$] Let $K^{\pm,c}$ and (resp. $K^{\pm,d}$) be the continuous part  (resp. purely discontinuous part) of $K^{\pm}$ such that $K^{\pm,c}=K^{\pm}-K^{\pm,d}$ with $K^{\pm,d}=\sum_{0<t\leq \cdot}\Delta K^{\pm}_t$. Then,  the Skorokhod condition $\int_{0}^{T}(Y_{t-}-L_{t-})dK^{+}_s=0$ is equivalent to $\int_{0}^{T}(Y_{t}-L_{t})dK^{+,c}_s=0$ and $\Delta K^{+,d}_t= \left(L_{t-}-Y_t\right)^{+}\mathds{1}_{\{\Delta L_t<0\}}= \left(L_{t-}-Y_t\right)^{+}\mathds{1}_{\{Y_{t-}=L_{t-}\}}$. In the same way we can obtain the property related to the reflection with respect to the upper obstacle $U$.
		\item[$\bullet$] The purely discontinuous processes $(K^{\pm,d}_t)_{t \leq T}$ has the following expression : $\forall t \leq T$,
		$$
		K^{+,d}_t=\sum_{0<s\leq t} (L_{s-}-Y_{s})^{+}\mathds{1}_{\{\Delta L_s<0\}\cap\{Y_{s-}=L_{s-}\}\cap\{\Delta Y_s<0\}},
		$$
		and
		$$
		K^{-,d}_t=\sum_{0<s\leq t} (Y_{s}-U_{s-})^{+}\mathds{1}_{\{\Delta U_s>0\}\cap\{Y_{s-}=U_{s-}\}\cap\{\Delta Y_s>0\}}.
		$$
	\end{itemize}
	
	For explicit detailts on the last two remark points, we refer readers to \cite[Remark 2.1]{ElmansouriOtmani2024}.
	\label{Reflection remark}
\end{remark}

We adopt the following definition of a solution to the DRBSDE \eqref{basic equation}.

\begin{definition}
	Let $\beta > 0$. A solution to the DRBSDE \eqref{basic equation} with jumps, associated with the data $(\xi, f, L, U)$, is a quintuple of processes $(Y, Z,,V, K^{+}, K^{-})$ that satisfies (\ref{basic equation}) and belongs to the space $\mathfrak{D}_{\beta}^2$.
	\label{Definition of the solution}
\end{definition}

It should also be emphasized that the DRBSDE \eqref{basic equation} may not admit a solution for arbitrary barriers $L$ and $U$. Indeed, if $L = U$ and $U$ is not a special semimartingale, it is impossible to find a semimartingale $Y$ satisfying \eqref{basic equation}-(i) that coincides with $U$. Therefore, it is necessary to impose additional conditions on the data, which are outlined in the following paragraph.
\paragraph{Conditions on the data $(\xi,f,L,U)$}
The quadruplet $(\xi,f,L,U)$ is such that:
\begin{itemize}
	\item[\textbf{(H1)}] The terminal variable $\xi$ is $\mathcal{F}_T$-measurable such that $\xi \in \mathbb{L}^2_{\beta}(\mathcal{F}_T)$.
	
	\item[\textbf{(H2)}] The driver $f :  \Omega \times [0,T] \times \mathbb{R} \times \mathbb{R} \times \ell^2 \rightarrow \mathbb{R}$ is such that
	\begin{itemize}
		\item[(i)] For all $(y,z)$, the stochastic process $f(\cdot,y,z,v)$ is $\mathbb{F}$-progressively measurable.
		\item[(ii)] There exist three non-negative $\mathbb{F}$-adapted processes $(\kappa_t)_{t \leq T}$,  $(\lambda_t)_{t \leq T}$ and $(\varrho_t)_{t \leq T}$ such that
		\begin{itemize}
			\item[(a)] for all $t \in [0,T]$, $y$, $y^{\prime} \in \mathbb{R}$ and $z$, $z^{\prime} \in \mathbb{R}$, $v$, $v' \in \ell^2$
			$$
			\left| f(t,y,z,v)-f(t,y^{\prime},z^{\prime},v') \right| \leq \kappa_t  \left|y-y^{\prime} \right|+\theta_t  \left|z-z^{\prime} \right|+\varrho_t \|(v-v') \circ \sqrt{\lambda_t}\|_{\ell^2}.
			$$
			\item[(b)] There exists $\epsilon >0$ such that $\alpha_t^2:=\kappa_t+\theta^2_t+\varrho_t^2 \geq \epsilon$ for all $t \in [0,T]$.
		\end{itemize}
		\item[(iii)] The process $\left(\frac{f(t,0,0,0)}{\alpha_{t}}\right)_{t \leq T}$ belongs to  $\mathcal{H}^2_{\beta}$.
	\end{itemize}
	\item[\textbf{(H3)}]  The barriers $(L_t)_{t \leq T}$ and $(U_t)_{t \leq T}$ are real-valued $\mathbb{F}$-progressively measurable RCLL processes satisfying
	\begin{itemize}
		\item[(i)] $L_T \leq \xi \leq U_T$, $\mathbb{P}$-a.s.
		\item[(ii)] $\mathbb{E}\left[ \sup_{0 \leq t \leq T} \left| e^{\beta A_t} L^{+}_t \right|^2 \right]<+\infty$ and $\mathbb{E}\left[ \sup_{0 \leq t \leq T} \left| e^{\beta A_t} U^{-}_t \right|^2 \right]<+\infty$,
		\item[(iii)] $L_t < U_t$ for all $t \in [0,T[$ and $L_{t-} < U_{t-}$ for all $t \in [0,T][$.
	\end{itemize}
	
	\item[\textbf{(H4)}] The complete separation of the barriers $L$ and $U$ and their left limits (\textbf{(H3)}-(iii)) can be strengthened by the existence of a semimartingale $\mathsf{Y} := (\mathsf{Y}_t)_{t \leq T}$ satisfying
	\begin{equation}
		\mathsf{Y}_t = \mathsf{Y}_0 + \int_0^t \mathsf{Z}_s dB_s + \sum_{k=1}^d \int_0^t \mathsf{V}^{(k)}_s d\widetilde{N}^{(k)}_s - \mathsf{K}^{+}_t + \mathsf{K}^{-}_t, \quad \mathsf{Y}_T = \xi,
		\label{Definition of the process J}
	\end{equation}
	and
	\begin{equation}
		L_t \leq \mathsf{Y}_t \leq U_t, \quad t \in [0, T],
		\label{Domina J}
	\end{equation}
	where $\mathsf{K}^{\pm}$ are two non-decreasing $\mathbb{F}$-adapted continuous processes with $\mathsf{K}_0^{\pm} = 0$. Moreover, we assume that $(\mathsf{Z}, \mathsf{V}) \in \mathcal{H}^2 \times \mathcal{V}^2_\lambda$ and $\mathbb{E}\left[ \left( \mathsf{K}_T^{\pm} \right)^2 \right] < +\infty$.
\end{itemize}

\begin{remark}
	The existence of a semimartingale $\mathsf{Y}$ of the form (\ref{Definition of the process J}) that satisfies (\ref{Domina J}) in \textbf{(H4)} has been established recently by Elmansouri \cite{Elmansouri2024PJM} within a more general filtration framework. In the particular case involving a Poisson random measure, we refer to the work of Hamadène and Wang \cite{HamadeneWang2009}. We also refer to the contribution by Hamadène et al. \cite{HamadeneHassaniOuknine2010}, and in particular to Remark 4.1, for further interesting discussions.
\end{remark}

\subsection{A priori estimates and uniqueness}
Let us first state an auxiliary result concerning the first triplet $(Y,Z,V)$ of the solution to the DRBSDE \eqref{basic equation}.

\begin{lemma}\label{Lemma 1}
	Let $(Y,Z,V) \in \mathfrak{B}^2_{\beta}$. Then, the stochastic integrals 
	$$
	\int_{0}^{\cdot} e^{\beta A_s} Y_s Z_s dB_s \quad \text{and} \quad \sum_{k=1}^d \int_{0}^{\cdot} e^{\beta A_s} Y_{s-} V^{(k)}_s d\widetilde{N}^{(k)}_s
	$$ 
	are uniformly integrable martingales with zero expectation.
\end{lemma}

\begin{proof}
	For both terms, it suffices to apply the Burkholder-Davis-Gundy inequality (BDG for short; see, e.g., \cite[Ch. IV, Theorem 48]{Protter2005}).
	
	\begin{itemize}
		\item For the Brownian integral $\int_{0}^{\cdot} e^{\beta A_s} Y_s Z_s dB_s$, we have:
		\begin{equation*}
			\begin{split}
				\mathbb{E}\left[\sup_{0 \leq t \leq T} \left|\int_{0}^{t} e^{\beta A_s} Y_s Z_s dB_s\right|\right] 
				&\leq c \mathbb{E}\left[\left(\int_{0}^{T} e^{2\beta A_s} |Y_s|^2 |Z_s|^2 ds \right)^{1/2}\right] \\
				&\leq \frac{1}{8} \| Y \|_{\mathcal{S}^2_{\beta}}^2 + 2c^2 \| Z \|_{\mathcal{H}^2_{\beta}}^2 < +\infty,
			\end{split}
		\end{equation*}
		where $c$ is the universal constant in the BDG inequality.
		
		\item Similarly, for the compensated inhomogeneous Poisson integral $\sum_{k=1}^d \int_{0}^{\cdot} e^{\beta A_s} Y_{s-} V^{(k)}_s d\widetilde{N}^{(k)}_s$, we have:
		\begin{equation*}
			\begin{split}
				&\mathbb{E}\left[\sup_{0 \leq t \leq T} \left|\sum_{k=1}^d \int_{0}^{t} e^{\beta A_s} Y_{s-} V^{(k)}_s d\widetilde{N}^{(k)}_s\right|\right] \\
				&\leq c \mathbb{E}\left[\left(\sum_{k,k'=1}^d \int_{0}^{T} e^{2\beta A_s} |Y_{s-}|^2 V^{(k)}_s V^{(k')}_s d\big[N^{(k)}, N^{(k')}\big]_s \right)^{1/2}\right] \\
				&= c \mathbb{E}\left[\left(\sum_{k=1}^d \int_{0}^{T} e^{2\beta A_s} |Y_{s-}|^2 |V^{(k)}_s|^2 \lambda^{(k)}_s ds \right)^{1/2}\right] \\
				&\leq \frac{1}{8} \| Y \|_{\mathcal{S}^2_{\beta}}^2 + 2c^2 \| V \|_{\mathcal{H}^2_{\lambda,\beta}}^2 < +\infty.
			\end{split}
		\end{equation*}
	\end{itemize}
	
	The above two estimates, together with \cite[Ch. I, Theorem  51]{Protter2005}, complete the proof.
\end{proof}

\begin{proposition}\label{Useful proposition}
	Assume given two $\mathbb{F}$-stopping times $\tau$, $\sigma$ such that $\tau \in \mathcal{T}_0^T$ and $\sigma \in \mathcal{T}_{\tau}^T$. Then, for any $\beta >2$, there exists a constant $\mathfrak{C}_{\beta}>0$ such that,
	\begin{equation*}
		\begin{split}
			&\mathbb{E}^{\mathcal{F}_\tau} \left[ \sup_{\tau \leq t \leq \sigma} e^{\beta A_t} \left| \overline{Y}_t \right|^2  \right]+ \mathbb{E}^{\mathcal{F}_\tau}\left[\int_\tau^\sigma e^{\beta A_s} \left| \overline{Y}_s \right|^2 dA_s \right]\\
			&\qquad+\mathbb{E}^{\mathcal{F}_\tau}\left[\int_\tau^\sigma e^{\beta A_s}  \left|\overline{Z}_s \right|^2 ds \right]+\mathbb{E}^{\mathcal{F}_\tau}\left[\int_\tau^\sigma e^{\beta A_s}   \|\overline{V}_{s}\|^2_\lambda ds \right] \\
			&  \leq \mathfrak{C}_{\beta}\Biggl\{ \mathbb{E}^{\mathcal{F}_\tau}\left[e^{\beta A_\sigma} \left| \overline{Y}_{\sigma} \right|^2\right]+\mathbb{E}^{\mathcal{F}_\tau}\left[\int_{\tau}^\sigma e^{\beta A_s} \left|\dfrac{ \bar{f}(s,Y^{\prime}_s,Z^{\prime}_s,V'_s)}{\alpha_s}\right|^2 ds \right]\\
			&\qquad +\mathbb{E}^{\mathcal{F}_\tau}\left[\int_{\tau}^{\sigma}e^{\beta A_s}\left\{ \bar{L}_{s-}^{+}dK^{+}_s+\bar{L}_{s-}^{-} dK^{\prime,+}_s \right\}\right] 
			+\mathbb{E}^{\mathcal{F}_\tau}\left[\int_{\tau}^{\sigma}e^{\beta A_s}\left\{\bar{U}_{s-}^{-}dK^{-}_s+\bar{U}_{s-}^{+}dK^{\prime,-}_s \right\}\right] \Biggr\}.
		\end{split}
	\end{equation*}
\end{proposition}

\begin{proof}
	Let $(Y,Z,V,K^+,K^-)$ and $(Y',Z',V',K^{\prime +},K^{\prime -})$ be two solutions of the DRBSDE \eqref{basic equation} associated with the data $(\xi,f,L,U)$ and $(\xi',f',L',U')$, respectively.\\ 
	For any
	$ 
	\mathcal{R}\in\{Y,Z,V,f,\xi,L,U,K^+,K^-\},
	$ 
	we use the notation
	$
	\overline{\mathcal{R}}:=\mathcal{R}-\mathcal{R}'.
	$

	Using Itô's formula (see, for instance, \cite[Ch II. Theorem 32]{Protter2005}), we can express
	\begin{equation}\label{basic Itos formula}
		\begin{split}
			&e^{\beta A_T}\left|\overline{\xi}\right|^2\\
			&=e^{\beta A_t}\left|\overline{Y}_t\right|^2+\beta\int_{t}^{T}e^{\beta A_s}\left|\overline{Y}_{s}\right|^2 dA_s+2\int_{t}^{T}e^{\beta A_s}\overline{Y}_{s-} d\overline{Y}_s+\int_{t}^{T}e^{\beta A_s}\left|\overline{Y}_{s-}\right|^2 d\big[\overline{Y},\overline{Y}\big]_s\\
			&=e^{\beta A_t}\left|\overline{Y}_t\right|^2+\beta\int_{t}^{T}e^{\beta A_s}\left|\overline{Y}_{s}\right|^2 dA_s-2\int_{t}^{T}e^{\beta A_s}\overline{Y}_{s}\left(f(s,Y_s,Z_s,V_s)-f'(s,Y'_s,Z'_s,V'_s)\right)ds\\
			&\qquad-2\int_{t}^{T}e^{\beta A_s}\overline{Y}_{s-}d\overline{K}^{+}_s+2\int_{t}^{T}e^{\beta A_s}\overline{Y}_{s-}d\overline{K}^{-}_s+2\int_{t}^{T}e^{\beta A_s}\overline{Y}_{s} \overline{Z}_sdB_s
			\\
			&\qquad+\sum_{k=1}^d \int_{t}^{T} e^{\beta A_s} Y_{s-} V^{(k)}_s d\widetilde{N}^{(k)}_s+\int_{t}^{T}e^{\beta A_s} \left|\overline{Z}_s\right|^2_s ds\\
			&\qquad+\sum_{k,k'=1}^d \int_{t}^{T}e^{\beta A_s}\overline{V}^{(k)}_s \overline{V}^{(k')}_s d\big[N^{(k)},N^{(k')}\big]_s
		\end{split}
	\end{equation}
	Next, using assumption \textbf{(H2)}-(ii) and  inequality $2ab \leq \varepsilon a^2+\frac{1}{\varepsilon}b^2$, for any $\varepsilon >0$, we get
	\begin{equation}
		\begin{split}
			2\overline{Y}_s (f(s,Y_s,Z_s,V_s)-f(s,Y^{\prime}_s,Z^{\prime}_s,V'_s))
			&\leq 2\kappa_s \left|\overline{Y}_s \right|^2+2 \gamma_s \left|\overline{Y}_s \right| \left| \overline{Z}_s\right|2+2 \varrho_s \left| \overline{Y}_s \right| \left\| \overline{V}_s \circ \sqrt{\lambda_s} \right\|_{\ell^2}\\
			& \leq 2 \alpha^2_s  \left| \overline{Y}_s \right|^2+\frac{1}{2}\left| \overline{Z}_s \right|^2+\frac{1}{2}\left\| \overline{V}_s \circ \sqrt{\lambda_s} \right\|^2_{\ell^2}
		\end{split}
		\label{ready to go 1}
	\end{equation}
	On the other hand,
	\begin{equation}
		\begin{split}
			2\overline{Y}_s \bar{f}(s,Y^{\prime}_s,Z^{\prime}_s,V'_s)-f^{\prime}(s,Y^{\prime}_s,Z^{\prime}_s,V'_s)&\leq 2\dfrac{\left| \bar{f}(s,Y^{\prime}_s,Z^{\prime}_s,V'_s)\right| }{\alpha_s}\alpha_s \left| \overline{Y}_s\right|\\
			& \leq  \varepsilon \alpha^2_s  \left| \overline{Y}_s\right|^2+\dfrac{1}{\varepsilon}\left| \dfrac{ \bar{f}(s,Y^{\prime}_s,Z^{\prime}_s,V'_s) }{\alpha_s}\right|^2
		\end{split}
		\label{ready to go 2}
	\end{equation}
	Furthermore, thanks to Skorokhod's condition (\ref{basic equation})-(iii), we obtain
	\begin{equation}
		\begin{split}
			&\int_{t}^{\sigma} e^{\beta A_s} \overline{Y}_{s-}d\overline{K}^{+}_s\\
			&=\int_{t}^{\sigma} e^{\beta A_s} \left(Y_{s-}-Y^{\prime}_{s-}\right)\left(dK^{+}_s-dK^{\prime,+}_s\right)\\
			&=\int_{t}^{\sigma} e^{\beta A_s} \left(Y_{s-}-L_{s-}\right)dK^{+}_s+\int_{t}^{\sigma} e^{\beta A_s} \left(L_{s-}-Y_{s-}\right)dK^{\prime,+}_s\\
			&\qquad+\int_{t}^{\sigma} e^{\beta A_s} \left(L^{\prime}_{s-}-Y^{\prime}_{s-}\right)dK^{+}_s+\int_{t}^{\sigma} e^{\beta A_s} \left(Y^{\prime}_{s-}-L^{\prime}_{s-}\right)dK^{\prime,+}_s+\int_{t}^{\sigma} e^{\beta A_s} \bar{L}_{s-}d\bar{K}^+_s \\
			&\leq \int_{t}^{\sigma} e^{\beta A_s} \bar{L}_{s-}d\bar{K}^+_s. 
		\end{split}
		\label{L minoration to get uniqueness}
	\end{equation}
	Similarly, we may show that
	\begin{equation}
		\int_{t}^{\sigma} e^{\beta A_s} \overline{Y}_{s-}d\bar{K}^-_s\geq \int_{t}^{\sigma} e^{\beta A_s} \bar{U}_{s-}d\bar{K}^-_s,~\mathbb{P}\text{-a.s.}
		\label{U majoration for uniq}
	\end{equation}
	Plugging (\ref{ready to go 1}), (\ref{ready to go 2}), (\ref{L minoration to get uniqueness}) and (\ref{U majoration for uniq}) into (\ref{basic Itos formula}), we obtain, after taking the conditional expectation on both sides along with Lemma \ref{Lemma 1},
	\begin{equation}
		\begin{split}
			&e^{\beta A_t} \left| \overline{Y}_{t} \right|^2+(\beta-\left(\varepsilon+2 \right) ) \mathbb{E}^{\mathcal{F}_t}\left[\int_t^\sigma e^{\beta A_s} \left| \overline{Y}_{s} \right|^2 dA_s \right]\\
			&\qquad+\dfrac{1}{2}\mathbb{E}^{\mathcal{F}_t}\left[\int_t^\sigma e^{\beta A_s}  \left| \overline{Z}_{s} \right|^2 ds \right]+\dfrac{1}{2}\mathbb{E}^{\mathcal{F}_t}\left[\int_t^\sigma e^{\beta A_s}   \|\overline{V}_{s}\|^2_\lambda ds \right]\\
			& \leq\mathbb{E}^{\mathcal{F}_t}\left[e^{\beta A_\sigma} \left| \overline{Y}_{\sigma} \right|^2\right]+\dfrac{1}{\theta}\mathbb{E}^{\mathcal{F}_t}\left[\int_{t}^\sigma e^{\beta A_s} \left|\dfrac{ \overline{f}(s,Y^{\prime}_s,Z^{\prime}_s,V'_s) }{\alpha_s}\right|^2 ds \right]\\
			&\qquad+2\mathbb{E}^{\mathcal{F}_t}\left[\int_{t}^{\sigma}e^{\beta A_s}\left\{ \bar{L}_{s-}^{+}dK^{+}_s+\bar{L}_{s-}^{-}dK^{\prime,+}_s \right\}\right]
			+2 \mathbb{E}^{\mathcal{F}_t}\left[\int_{t}^{\sigma}e^{\beta A_s}\left\{\bar{U}^{-}_{s-}dK^{-}_s+\bar{U}_{s-}^{+}dK^{\prime,-}_s \right\}\right],
		\end{split}
	\end{equation}
	Choosing $\varepsilon>0$ such that $\beta >\varepsilon+2$ and taking conditional expectation with respect to $\mathcal{F}_{\tau}$, we obtain
	\begin{equation}
		\begin{split}
			&\sup_{\tau \leq t \leq \sigma}\mathbb{E}^{\mathcal{F}_\tau}\left[e^{\beta A_t} \left| \overline{Y}_t \right|^2\right]+ \mathbb{E}^{\mathcal{F}_\tau}\left[\int_\tau^\sigma e^{\beta A_s}  \left| \overline{Y}_s \right|^2 dA_s \right]\\
			& \qquad+\mathbb{E}^{\mathcal{F}_\tau}\left[\int_\tau^\sigma e^{\beta A_s}  \left| \overline{Z}_s\right|^2 ds \right]+\mathbb{E}^{\mathcal{F}_\tau}\left[\int_\tau^\sigma e^{\beta A_s}   \|\overline{V}_{s}\|^2_\lambda ds \right]\\
			& \leq \mathfrak{C}_{\beta}\Biggl\{ \mathbb{E}^{\mathcal{F}_\tau}\left[e^{\beta A_\sigma} \left| \overline{Y}_{\sigma} \right|^2\right]+\mathbb{E}^{\mathcal{F}_\tau}\left[\int_{\tau}^\sigma e^{\beta A_s} \left|\dfrac{ \overline{f}(s,Y^{\prime}_s,Z^{\prime}_s,V'_s)}{\alpha_s}\right|^2 ds \right]\\
			&\qquad +\mathbb{E}^{\mathcal{F}_\tau}\left[\int_{\tau}^{\sigma}e^{\beta A_s}\left\{ \bar{L}_{s-}^{+}dK^{+}_s+\bar{L}_{s-}^{-}dK^{\prime,+}_s \right\}\right] 
			+\mathbb{E}^{\mathcal{F}_\tau}\left[\int_{\tau}^{\sigma}e^{\beta A_s}\left\{\bar{U}^{-}_{s-}dK^{-}_s+\bar{U}_{s-}^{+}dK^{\prime,-}_s  \right\}\right]\Biggr\}.
		\end{split}
		\label{Z,N estimates}
	\end{equation}
	Finally, using once again the BDG inequality as in Lemma \ref{Lemma 1} along with the above estimation, we derive
	\begin{equation}
		\begin{split}
			&\mathbb{E}^{\mathcal{F}_\tau} \left[ \sup_{\tau \leq t \leq \sigma} e^{\beta A_t} \left| \overline{Y}_t \right|^2  \right] \\
			&  \leq \mathfrak C_{\beta}\Biggl\{ \mathbb{E}^{\mathcal{F}_\tau}\left[e^{\beta A_\sigma} \left| \overline{Y}_{\sigma} \right|^2\right]+\mathbb{E}^{\mathcal{F}_\tau}\left[\int_{\tau}^\sigma e^{\beta A_s} \left|\dfrac{ \bar{f}(s,Y^{\prime}_s,Z^{\prime}_s,V'_s)}{\alpha_s}\right|^2 ds \right]\\
			&\qquad +\mathbb{E}^{\mathcal{F}_\tau}\left[\int_{\tau}^{\sigma}e^{\beta A_s}\left\{ \bar{L}_{s-}^{+}dK^{+}_s+\bar{L}_{s-}^{-} dK^{\prime,+}_s \right\}\right] 
			+\mathbb{E}^{\mathcal{F}_\tau}\left[\int_{\tau}^{\sigma}e^{\beta A_s}\left\{\bar{U}_{s-}^{-}dK^{-}_s+\bar{U}_{s-}^{+}dK^{\prime,-}_s \right\}\right] \Biggr\}.
		\end{split}
	\end{equation}
	
	The result of Proposition \ref{Useful proposition} is achieved.
\end{proof}

From Proposition \ref{Useful proposition}, we derive the following uniqueness result.
\begin{corollary}\label{Uniq}
	Assume \textbf{(H1)}--\textbf{(H3)}. Then, there exists at most one solution $(Y,Z,V,K^{+},K^{-})$ of the DRBSDE (\ref{basic equation}) associated with $(\xi,f,L,U)$. 
\end{corollary}

\begin{proof}
	Let $(Y,Z,V,K^+,K^-)$ and $(Y',Z',V',K^{\prime +},K^{\prime -})$ be two solutions of the DRBSDE \eqref{basic equation} associated with the same data $(\xi,f,L,U)$.
	
	By the a priori estimates established in Proposition \ref{Useful proposition}, we obtain
	\[
	\begin{aligned}
		&\mathbb{E}\left[
		\sup_{0\leq t\leq T}
		e^{\beta A_t}|\overline{Y}_t|^2
		\right]
		+
		\mathbb{E}\left[
		\int_0^T e^{\beta A_s}
		|Z_s-Z'_s|^2\,ds
		\right] \\
		&\qquad
		+
		\mathbb{E}\left[
		\int_0^T e^{\beta A_s}
		\|V_s-V'_s\|_\lambda^2\,ds
		\right]
		=0.
	\end{aligned}
	\]
	Consequently,
	\[
	(Y,Z,V)=(Y',Z',V').
	\]
	
	Next, by the second item of Remark \ref{Reflection remark}, the expressions of
	$K^{\pm,d}$ and $K^{\prime\pm,d}$ in terms of $Y$ and $Y'$ imply that
	\[
	K^{\pm,d}=K^{\prime\pm,d}.
	\]
	For the continuous parts, the complete separation assumption
	\textbf{(H3)}-(iii), together with the Skorokhod conditions, yields
	\[
	(U-L)\bigl(dK^{c,\pm}-dK^{\prime c,\pm}\bigr)=0.
	\]
	Since $U-L>0$, it follows that
	\[
	K^{c,\pm}=K^{\prime c,\pm}.
	\]
	Therefore,
	\[
	K^\pm=K^{\prime\pm},
	\]
	which completes the proof.
\end{proof}

\subsection{Existence via Penalization Approximation}
In this section, we establish the existence of solutions for the DRBSDE \eqref{basic equation} using the penalization method. This technique not only proves existence but also provides an approximation of the solution via a sequence of standard BSDEs of a particular form.\\

The proof is divided into two main steps. First, we treat the case where the driver $f$ is independent of $(y,z,v)$, that is, $f(\omega,t,y,z,v) = g(\omega,t)$ for any $(t,y,z,v) \in [0,T] \times \mathbb{R}^2 \times \mathcal{V}^2_\lambda$, $\mathbb{P}$-a.s. In this setting, we have $\frac{g(\cdot)}{\alpha} \in \mathcal{H}^2_\beta$ from \textbf{(H2)}-(iii), then we construct a sequence of penalized equations whose solutions converge to the solution of the DRBSDE (\ref{basic equation}) associated with the data $(\xi, g, L, U)$. The general case will then be handled using a fixed-point argument in a suitable Banach space.

\subsubsection{Case where the driver $f$ does not depend on $(y, z, v)$}
In this first part, we prove the following result:
\begin{theorem}
	Suppose that assumptions \textbf{(H1)}--\textbf{(H4)} hold for a sufficiently large $\beta > 0$. Then, the DRBSDE (\ref{basic equation}) associated with the data $(\xi, g, L, U)$ admits a unique solution $\left(Y_t, Z_t, V_t, K^{+}_t, K^{-}_t\right)_{t \leq T}$ in the space $\mathfrak{D}^{2}_{\beta}$.
	\label{Existence and uniquenss: g}
\end{theorem}

The uniqueness has already been proved in Corollary \ref{Uniq}.\\
The existence result for the DRBSDE (\ref{basic equation}) associated with $(\xi, g, L, U)$ via the penalization method follows by adapting the arguments from \cite[Theorem 4.1]{ElmansouriOtmani2024}. For conciseness, we present only the main steps of the proof and refer the reader to \cite{ElmansouriOtmani2024} for complete details.

\begin{proof}
	We consider the following penalization approximation schemes with respect to the two reflecting barriers $L$ and $U$, defined for each $n \in \mathbb{N}$ as follows:
	\begin{equation}
		\begin{split}
			Y^{n}_t= &\xi+\int_t^T g(s)ds+n\int_t^T (L_s-Y^{n}_s)^{+}ds  -n \int_t^T (Y^{n}_s-U_s)^{+}ds\\
			&\qquad  -\int_t^T Z^{n}_s d B_s-\sum_{k=1}^d\int_{t }^{T}V^{n,(k)}_s d \widetilde{N}^{(k)}_s, \quad t \in [0, T].\\
		\end{split}
		\label{penalization equations with respect to the two reflection barriers}
	\end{equation}
	We denote $K^{n,+}_t := n\int_0^t (L_s - Y^{n}_s)^{+} ds$, $K^{n,-}_t := n\int_0^t (Y^{n}_s - U_s)^{+} ds$, and $f_n(s,y) := g(s) + n\left(L_s - y\right)^{+} - n\left(y - U_s\right)^{+}$.\\ 
	Under assumptions \textbf{(H2)}-(ii)-(iii) and \textbf{(H3)}-(ii), we have
	\begin{equation*}
		\begin{split}
			&\mathbb{E}\left[\int_{0}^{T} e^{\beta A_s} \left|\frac{f_n(s,0)}{\alpha_s}\right|^2 ds \right]\\
			&\leq 3\left(\mathbb{E}\left[\int_{0}^{T} e^{\beta A_s} \left|\frac{g(s)}{\alpha_s}\right|^2 ds \right] + T\left\{ \mathbb{E}\left[ \sup_{0 \leq t \leq T} \left| e^{\beta A_t} L^{+}_t \right|^2 \right] + \mathbb{E}\left[ \sup_{0 \leq t \leq T} \left| e^{\beta A_t} U^{-}_t \right|^2 \right] \right\} \right)\\
			&< +\infty.
		\end{split}
	\end{equation*}
	Moreover, the driver $f_n$ is a $2n$-Lipschitz mapping. Then, using \cite[Theorem 3.5]{JamaliOtmani2019}, \cite[Appendix A]{ElmansouriOtmani2024}, or \cite[Theorem 9]{ElJamaliOtmani2022}, we deduce that there exists a unique triplet of processes $\left(Y^n_t, Z^n_t, V^n_t\right)_{t \leq T}$ belonging to $\mathfrak{B}^2_{\beta}$ that satisfies the classical BSDE (\ref{penalization equations with respect to the two reflection barriers}) associated with parameters $(\xi, f_n)$.\\
	Next, using assumption \textbf{(H4)} along with a localization procedure, as in the proof of Lemma 4.1 in \cite{ElmansouriOtmani2024} or Step 2 of the proof of Theorem 3 in \cite{MarzougueOtmani2021}, we can derive uniform \textit{a priori} estimates for the sequence of processes $\{(Y^n, Z^n, V^n, K^{n,+}, K^{n,-})\}_{n \in \mathbb{N}}$, ensured by the existence of a constant $\mathfrak{C}_{\beta,T}$ such that
	\begin{equation}\label{UI}
		\begin{split}
			&\left\| Y^n\right\|^2_{\mathcal{S}^2_{\beta}}   +\left\| Y^n\right\|^2_{\mathcal{S}^{2,\alpha}_\beta}    +\left\| Z^n\right\|^2_{\mathcal{H}^2_{\beta}} +\left\| V^n \right\|^2_{\mathcal{V}^2_{\gamma,\beta}} +\left\| K^{n,+}_T \right\|^2_{\mathcal{K}^2}+\left\| K^{n,-}_T \right\|^2_{\mathcal{K}^2} \\
			&\leq \mathfrak{C}_{\beta,T} \Biggl\{\left\| \xi\right\|^2_{\mathcal{L}^2_{\beta}(\mathcal{G}_T)}  +\left\| \dfrac{g(\cdot)}{\alpha_{\cdot}}\right\|^2_{\mathcal{H}^2_{\beta}} + \left\| L^{+}\right\|^2_{\mathcal{S}^2_{2\beta}}+ \left\| U^{-}\right\|^2_{\mathcal{S}^2_{2\beta}}+\left\| \mathsf{Z}\right\|^2_{\mathcal{H}^2} +\left\| \mathsf{V}\right\|^2_{\mathcal{V}^2_\gamma}\\
			&\qquad\qquad\qquad+\left\| \mathsf{K}^{+}_T \right\|^2_{\mathcal{K}^2}+\left\| \mathsf{K}^{-}_T \right\|^2_{\mathcal{K}^2} \Biggr\}.
		\end{split}
	\end{equation}
	
	We next aim to prove that the sequence $\{(Y^n, Z^n, V^n, K^n := K^{n,+} - K^{n,-})\}_{n \in \mathbb{N}}$ is a Cauchy sequence in the space $\mathcal{S}^2 \times \mathcal{H}^2 \times \mathcal{V}^2_\gamma \times \mathcal{K}^2$. To this end, we employ the following convergence result:
	\begin{equation}\label{CVV}
		\lim\limits_{n \rightarrow +\infty}\mathbb{E}\left[\sup_{0 \leq t \leq T} \left| \left(Y^n_t - L_t\right)^{-} \right|^2 + \sup_{0 \leq t \leq T} \left| \left(Y^n_t - U_t\right)^{+} \right|^2 \right] = 0.
	\end{equation}
	Note that this convergence result is established in the Brownian setting in \cite[Lemma 3.3]{LiShi2016}, in the proof of Step 3 of Theorem 3 in \cite{MarzougueOtmani2021} for the particular case of a homogeneous Poisson process, in Step 3 of the proof of Theorem 4.1 in \cite{OtmaniJamaliMarzougue2022} for the case of a Lévy process, and is also detailed in the proof of Step 3 of Theorem 4.1 in \cite{ElmansouriOtmani2024} for the general case of a given RCLL martingale in a general filtration.
	
	Next, for each $n \geq p \geq 0$, applying It\^{o}'s formula as in Proposition \ref{Useful proposition} implies that
	\begin{equation*}
		\begin{split}
			&\mathbb{E}\left[\left|Y^n_t-Y^p_t \right|^2 \right]+\mathbb{E}\left[\int_{t}^{T}\left|Z^n_s-Z^p_s \right|^2 ds \right]+\mathbb{E}\left[\int_{t}^{T} \|V^n_s-V^p_s\|^2_\lambda ds   \right]\\
			&\leq 2\mathbb{E}\left[\int_{t}^{T} \left(Y^n_{s-}-Y^p_{s-}\right)\left(dK^{n,+}_s-dK^{p,+}_s\right) \right]-2\mathbb{E}\left[\int_{t}^{T} \left(Y^n_{s-}-Y^p_{s-}\right)\left(dK^{n,-}_s-dK^{p,-}_s\right)\right] \\
			&\leq 2\mathbb{E}\left[\sup_{0  \leq t \leq  T} \left|\left(  Y^n_t-L_t \right)^{-}  \right| K^{p,+}_T \right]+2\mathbb{E}\left[\sup_{0  \leq t \leq  T} \left|\left(  Y^p_t-L_t \right)^{-}  \right| K^{n,+}_T \right]\\
			&+2\mathbb{E}\left[\sup_{0  \leq t \leq  T} \left|\left(  Y^n_t-U_t \right)^{+}  \right| K^{p,-}_T \right]+2\mathbb{E}\left[\sup_{0  \leq t \leq  T} \left|\left(  Y^p_t-U_t \right)^{-}  \right| K^{n,-}_T \right].
		\end{split}
	\end{equation*}
	Using the uniform estimate (\ref{UI}) and the convergence result (\ref{CVV}), we deduce the existence of a limiting process $(Y, Z, V, K) \in \mathcal{S}^2 \times \mathcal{H}^2 \times \mathcal{V}^2_\gamma \times \mathcal{K}^2$ such that
	\begin{equation}\label{CV}
		\lim\limits_{n \rightarrow +\infty} \left( \left\| Y^n - Y \right\|^2_{\mathcal{S}^2} + \left\| Z^n - Z \right\|^2_{\mathcal{H}^2} + \left\| V^n - V \right\|^2_{\mathcal{V}^2_\gamma} + \left\| K^n - K \right\|^2_{\mathcal{K}^2} \right) = 0.
	\end{equation}
	Then, passing to the limit term by term in $\mathbb{L}^2\left(\Omega, d\mathbb{P}\right)$ as $n \rightarrow +\infty$ in (\ref{penalization equations with respect to the two reflection barriers}), we obtain
	\begin{equation}
		Y_t = \xi + \int_{t}^{T} g(s)\, ds + \left(K_T - K_t\right) - \int_{t}^{T} Z_s\, dB_s - \sum_{k=1}^d\int_{t }^{T}V^{(k)}_s d \widetilde{N}^{(k)}_s,
		\quad t \in [0,T].
		\label{equation verified by Y}
	\end{equation}
	
	Finally, to verify the Skorokhod conditions, we use the uniform estimate (\ref{UI}) for the sequences $\{K^{n,+}_\vartheta\}_{n \in \mathbb{N}}$ and $\{K^{n,-}_\vartheta\}_{n \in \mathbb{N}}$. Combined with the Hilbert space structure of $\mathbb{L}^2\left(\Omega, d\mathbb{P}\right)$, this allows us to extract a subsequence that converges weakly in the corresponding space to some $\mathcal{F}_{\vartheta}$-measurable random variable $K^{\pm}_{\vartheta}$ for each $\vartheta \in \mathcal{T}_{0}^T$.\\
	Next, by setting $\bar{K}_{\vartheta} = K^{+}_{\vartheta} - K^{-}_{\vartheta}$ and applying Mazur's Theorem (see Theorem 2 in \cite{Yosida2012}, p.~120), we obtain that for any $n \in \mathbb{N}$, there exists an integer $q(n) \geq n$, a set of weights $(\gamma^{(\vartheta,n)}_j)_{j \in \{n, \dots, q(n)\}} \subset \mathbb{R}^{+}$, and a convex combination $\sum_{j=n}^{q(n)} \gamma^{(\vartheta,n)}_j K^{j ,\pm}_{\vartheta}$ such that $\sum_{j=n}^{q(n)} \gamma^{(\vartheta,n)}_j = 1$, and
	$$
	\bar{K}^{n,\pm}_{\vartheta} := \sum_{j=n}^{q(n)} \gamma^{(\vartheta,n)}_j K^{j ,\pm}_{\vartheta} \xrightarrow[n \rightarrow +\infty]{} K^{\pm}_{\vartheta} \quad \text{in } \mathbb{L}^2(\Omega, d\mathbb{P}).
	$$
	By denoting $\bar{K}_{\vartheta}^n = \bar{K}_{\vartheta}^{n,+} - \bar{K}_{\vartheta}^{n,-}$ and using (\ref{CV}), we derive that  
	$$
	\lim\limits_{n \rightarrow +\infty} \mathbb{E}\left[ \left| \bar{K}_{\vartheta}^n - \bar{K}_{\vartheta} \right|^2 \right] = 0 
	\quad \text{and} \quad 
	\lim\limits_{n \rightarrow +\infty} \mathbb{E}\left[ \left| \bar{K}_{\vartheta}^n - K_{\vartheta} \right|^2 \right] = 0,
	$$
	which yields $\bar{K}_t = K_t$ for all $t \in [0,T]$, $\mathbb{P}$-a.s. \\
	Using again (\ref{UI}) and the selection principle (see Theorem 4.3.3 in \cite{Chung2001}, p.~88), we deduce that there exists a subsequence of $\{\bar{K}^{n,+}_{T}(\omega)\}_{n \in \mathbb{N}}$ (resp. $\{\bar{K}^{n,-}_{T}(\omega)\}_{n \in \mathbb{N}}$) that converges weakly to $K_{T}^{+}(\omega)$ (resp. $K_{T}^{-}(\omega)$).\\
	Set $L^{\xi}_t:=L_t \mathds{1}_{\{t< T\}}+\xi \mathds{1}_{\{t=T\}}$ and $U^{\xi}_t:=U_t \mathds{1}_{\{t< T\}}+\xi \mathds{1}_{\{t=T\}}$. By the notion of the Snell envelope, we know that
	\begin{equation*}
		Y_t -\mathbb{E}^{\mathcal{F}_{t}}\left[\xi-\int_{t}^{T} g(s)ds \right]=\mathfrak{R}_t(\pi^{+})-\mathfrak{R}_t(\pi^{-}),
	\end{equation*}
	where $\mathfrak{R}_{\cdot}$ denotes the Snell envelope operator with
	\begin{equation*}
		\begin{split}
			\pi^{+}_t&:=L^{\xi}_t-	\mathbb{E}^{\mathcal{F}_{t}}\left[\xi-\left(  K^{-}_{T}-K^{-}_{t }  \right)   \right];\qquad \pi^{+}_{T}=0,\\
			\pi^{-}_t&:=	-U^{\xi}_t+\mathbb{E}^{\mathcal{F}_{t}}\left[\xi+\left(  K^{+}_{T}-K^{+}_{t}\right)    \right];\quad \pi^{-}_{T}=0.
		\end{split}
	\end{equation*}
	Since $(\mathfrak{R}_t(\pi^{\pm}))_{t \leq T}$ are potentials satisfying $\mathbb{E}[\sup_{0 \leq s \leq T} \left| \mathfrak{R}_s(\pi^{\pm}) \right|^2]<\infty$, it follows from the Doob–Meyer decomposition of supermartingales and the uniqueness of the solution, along with results from optimal stopping theory relating to Snell envelopes (see, for example, the work by El Karoui \cite{ElKaroui1981} or by Kobylanski and Quenez \cite{Kobylanski-Quenez2012}), completes the proof of the first part. Finally, the integrability property for the triplet $(Y, Z, V)$ is obtained by applying Fatou's Lemma to the uniform estimate (\ref{UI}) and using the convergence result (\ref{CV}) to pass to the limit along a subsequence.
\end{proof}

\subsubsection{General case}
We are now ready to present the main result of this section. In this setting, the generator $f$ is general and may depend on both variables $(y, z, v)$.
\begin{theorem}\label{basic Thm}\footnote{The first author, \textbf{Badr Elmansouri}, would like to express his sincere gratitude to Professor \textbf{Youssef Ouknine} for his insightful remarks concerning the regularity of the solution when applying the fixed-point argument in an appropriate Banach space for the state variable.}
	Assume that \textbf{(H1)}--\textbf{(H4)} are satisfied for a sufficiently large value of $\beta$. Then, the DRBSDE (\ref{basic equation}) associated with $\left(\xi, f, L, U\right)$ admits a unique solution $\left(Y_t, Z_t, V_t, K^{+}_t, K^{-}_t\right)_{t \leq T}$ belonging to $\mathfrak{D}^2_{\beta}$.
\end{theorem}

\begin{proof}
	Let us consider the Banach space $\mathfrak{B}^2_{\beta}=	\left(
	\mathcal{S}^2_{\beta}
	\cap
	\mathcal{S}^{2,\alpha}_{\beta}
	\right)
	\times
	\mathcal{H}^2_{\beta}
	\times
	\mathcal{V}^2_{\lambda,\beta}
	$, endowed with the natural norm
	$$
	\left\|\left(Y,Z,V\right) \right\|_{\beta}=\left(\mathbb{E}\left[\sup_{0 \leq t \leq T} e^{\beta A_t} \left| {Y}_t \right|^2+\int_{0}^{T}e^{\beta A_s}\left(\left| Y_s \alpha_s\right|^2+\left| Z_s \right|^2+  \|V_s \circ \sqrt{\lambda_s}\|^2_{\ell^2}\right) ds\right] \right)^{\frac{1}{2}}.
	$$
	Using Theorem \ref{Existence and uniquenss: g}, we define the mapping $\Psi$ as follows. Let $\Psi$ be a mapping from $\mathfrak{B}^2_{\beta}$ into itself that associates with each triple $\left(X, W, P\right)$ the corresponding triple $\left(Y, Z, V\right)$, where $\left(Y, Z, V\right)$ is the solution of the DRBSDE \eqref{basic equation} associated with $\left(\xi, f\left(t, X_t, W_t, P_t\right), L, U\right)$. Let $\left(X^{\prime}, W^{\prime}, P^{\prime}\right)$ be another element of $\mathfrak{B}^2_{\beta}$, and set $\left(Y^{\prime}, Z^{\prime}, V^{\prime}\right) := \Psi\left(X^{\prime}, W^{\prime}, P^{\prime}\right)$. \
	By Theorem \ref{Existence and uniquenss: g}, the processes $Y$ and $Y'$ admit RCLL modifications.\\
	We now define $\overline{\mathcal{R}} := \mathcal{R} - \mathcal{R}^{\prime}$, for each $\mathcal{R} \in \{Y, Z, K^{+}, K^{-}, V, X, W, P\}$.
	
	Applying It\^{o}'s formula and using the Skorokhod condition \eqref{basic equation}-(iii) for the reflecting processes $\overline{K}^+$ and $\overline{K}^-$, namely, $\overline{Y}_s d\overline{K}^+_s -\overline{Y}_s d\overline{K}^-_s \leq 0$ a.s., we obtain, for any $t \leq T$ and $\beta>1$,
	\begin{equation}\label{inq1}
		\begin{split}
			&\mathbb{E}\left[\int_{0}^{T}e^{\beta A_s} \left|\alpha_s \overline{Y}_s \right|^2 d\left\langle M\right\rangle_s
			\right]+\mathbb{E}\left[\int_{0}^{T}e^{\beta A_s} \left| \bar{Z}_s \right|^2 ds
			\right]+\mathbb{E}\left[\int_{0}^{T}e^{\beta A_s}  \|\overline{V}_s\circ \sqrt{\lambda_s}\|^2_{\ell^2} ds
			\right]\\
			& \leq \dfrac{3}{ \beta -1 } \mathbb{E}\left[\int_{0}^{T}e^{\beta A_s}\left( \left\{\left|\alpha_s \overline{X}_s \right|^2+\left|\overline{W}_s \right|^2+\|\overline{P}_s \circ \sqrt{\lambda_s}\|^2_{\ell^2} \right\}  ds\right)  \right] .
		\end{split}
	\end{equation}

Next, using arguments similar to those employed in Lemma \ref{Lemma 1} and Proposition \ref{Useful proposition}, together with estimate \eqref{inq1}, we obtain the following estimates:
	\begin{equation}\label{inq2}
	\begin{split}
		&\mathbb{E} \left[ \sup_{0 \leq t \leq T} e^{\beta A_t} \left| \overline{Y}_t \right|^2  \right]\\
		&  \leq \dfrac{12 c^2+3}{ \beta -1 } \mathbb{E}\left[ \sup_{0 \leq t \leq T} e^{\beta A_t} \left| \overline{X}_t \right|^2+\int_{0}^{T}e^{\beta A_s}\left( \left\{\left|\alpha_s \overline{X}_s \right|^2+\left|\overline{W}_s \right|^2+\|\overline{P}_s \circ \sqrt{\lambda_s}\|^2_{\ell^2} \right\}  ds\right)  \right] .
	\end{split}
\end{equation}
By choosing $\beta > 12c^2 + 7$ and combining \eqref{inq1} and \eqref{inq2}, we obtain
$$
\left\|\left(\bar Y, \bar Z, \bar V\right) \right\|_{\beta} \leq \mathfrak{c} \left\|\left(\bar X, \bar W, \bar P\right) \right\|_{\beta}
$$
where $c\in (0,1)$. Therefore, the mapping $\Psi$ is a strict contraction on the Banach space $\mathfrak{B}^2_{\beta}$. Consequently, there exists a unique triple of processes $\left(Y_t, Z_t, V_t\right)_{t \leq T}$ that is a fixed point of $\Psi$, i.e., $\Psi\left(Y, Z, V\right) = \left(Y, Z, V\right)$. Together with $K^{+}$ and $K^{-}$, this triple forms the unique solution of the DRBSDE (\ref{basic equation}) associated with $\left(\xi, f, L, U\right)$.

\end{proof}

\section{Comparison principal}
\label{sec3}
The comparison theorem is one of the principal tools in the theory of BSDEs. It is well known, however, that this result does not hold in general in the presence of jumps (see the counterexample in Barles et al. \cite{BarlesBuckdahnPardoux1997}). 

In order to establish a comparison theorem in this setting, and following the approach of \cite{JamaliOtmani2019}, \cite{DumitrescuQuenezSulem2018}, \cite{KrusePopier2016}, and \cite{Royer2006}, we impose a monotonicity condition on the generator $f$ with respect to the $v$-variable, along with an additional boundedness assumption on the random variable $A_T$.
\begin{itemize}
	\item[\textbf{(H5)}] Assume that $A_T=\int_{0}^{T} \left\{ \kappa_s + \theta_s^2+\varrho^2_s \right\} ds$ is a bounded random variable, and that there exists a map
	\begin{align*}
		\zeta : \Omega \times [0,T] \times \mathbb{R}^2 \times \ell^2 \times \ell^2 
		&\rightarrow \mathbb{R}^d; \\
		(\omega, t, y, z, v_1, v_2) &\mapsto \zeta^{y,z,v_1,v_2}_t(\omega)
		:= \left( \zeta^{y,z,v_1,v_2;(1)}_t(\omega), \dots, \zeta^{y,z,v_1,v_2;(d)}_t(\omega) \right)
	\end{align*}
	that is $\mathcal{P} \otimes \mathcal{B}(\mathbb{R}^2) \otimes \mathcal{B}(\ell^2) \otimes \mathcal{B}(\ell^2)$-measurable, and satisfies the following properties:
	\begin{equation}\label{Mono-}
		\left\lbrace 
		\begin{aligned}
			&\big| \zeta^{y,z,v_1,v_2;(k)}_t \big| \sqrt{\lambda^{(k)}_t} \leq  \varrho_t , ~ \text{ for each } k \in \{1, \dots, d\}, \text{ and all } t \in [0,T], \text{ a.s.}; \\
			&\sum_{k=1}^d \zeta^{y,z,v_1,v_2;(k)}_t > -1, \quad d\mathbb{P} \otimes dt\text{-a.e.}
		\end{aligned}
		\right.
	\end{equation}
Moreover, $d\mathbb{P} \otimes dt$-a.s., for all $(y, z, v_1, v_2) \in \mathbb{R}^2 \times \ell^2 \times \ell^2$, we have:
\begin{equation}\label{Mono}
	f(t, y, z, v_1) - f(t, y, z, v_2) \leq \sum_{k=1}^d \zeta^{y,z,v_1,v_2;(k)}_t (v^{(k)}_1 - v^{(k)}_2) {\lambda^{(k)}_t}.
\end{equation}
\end{itemize}

Under the additional assumption \textbf{(H5)}, we make the following observation, which leads to a version of assumption \textbf{(H2)}-(ii)-(a) formulated with respect to the variable $v$:
\begin{remark}\label{MonoV}
\begin{itemize}
	\item If (\ref{Mono}) holds, then by exchanging the roles of $v_1$ and $v_2$ in $\zeta$, we obtain
	$$
	f(t, y, z, v_1) - f(t, y, z, v_2) \geq \sum_{k=1}^d \zeta^{y,z,v_2,v_1;(k)}_t (v^{(k)}_1 - v^{(k)}_2) {\lambda^{(k)}_t}=\big\langle \zeta^{y,z,v_2,v_1}_t \circ \sqrt{\lambda^{}_t}, (v_1 - v_2) \circ \sqrt{\lambda^{}_t} \big\rangle_{\ell^2} .
	$$
	Applying the Cauchy–Schwarz inequality and using (\ref{Mono-}), we deduce
	$$
	\left| f(t, y, z, v_1) - f(t, y, z, v_2) \right| \leq \|\zeta_t\|_{\ell^2} \| (v - v') \circ \sqrt{\lambda_t} \|_{\ell^2} \leq d \varrho^2_s \| (v - v') \circ \sqrt{\lambda_t} \|_{\ell^2}.
	$$
	
	\item We emphasize that assumption \textbf{(H5)} has been previously employed in several contexts, including filtrations generated by Brownian–Poisson processes \cite{KrusePopier2016,Royer2006}, in defaultable frameworks \cite{DumitrescuQuenezSulem2018,D10}, and also in more general settings as in \cite{Nie2021}.
\end{itemize}
\end{remark}

Now we are in position to state the main result of this section.
\begin{theorem}\label{Compa}
Let $\left(Y^j, Z^j, V^j, K^{+, j}, K^{-, j}\right)$ be the unique solution of the DRBSDE (\ref{basic equation}) associated with data $\left(\xi^j, f^j, L^j, U^j\right)$, for $j=1,2$. \\
Assume that $f^1\left(., Y^2, Z^2, V^2\right) \leq f^2\left(., Y^2, Z^2, V^2\right)$ a.s., $\xi^1 \leq \xi^2$ a.s. and $L^1 \leq L^2$, $U^1 \leq U^2$ a.s.

Then $Y^1 \leq Y^2$ a.s.
\end{theorem}
\begin{proof}
Here, we follow the arguments used in the proof of Theorem 2.5 in \cite{Royer2006}, as well as those of Proposition 4 in \cite{KrusePopier2016}, originally established for classical BSDEs with jumps, and adapt them to our doubly reflected setting.

As usual, we define $\overline{\mathcal{R}} := \mathcal{R}^1 - \mathcal{R}^2$ for each $\mathcal{R} \in \{Y, Z, V, \xi, f, L, U\}$, and $\bar{\mathcal{R}} := \mathcal{R}^1 - \mathcal{R}^2$ for each $\mathcal{R} \in \{K^{+}, K^{-}\}$. Then, the quintuple $\left(\overline{Y}, \overline{Z}, \overline{V}, \bar{K}^+, \bar{K}^-\right)$ satisfies
\begin{equation}\label{Gs}
	\overline{Y}_{t}=\overline{\xi}+\int_{t}^{T} g(s)ds+\left( \bar{K}^{+}_{ T}-\bar{K}^{+}_{t}\right)-\left(\bar{K}^{-}_{ T}-\bar{K}^{-}_{t }\right)
	-\int_{t}^{T} \overline{Z}_s d B_s-\sum_{k=1}^d\int_{t }^{T}\overline{V}^{(k)}_s d \widetilde{N}^{(k)}_s,
\end{equation}
where $g(s)=f^1(s,Y^1_s,Z^1_s,V^1_s)-f^2(s,Y^2_s,Z^2_s,V^2_s)$.\\ 
Now, we define
\begin{equation*}
	\left\lbrace 
	\begin{split}
		f_s:&=f^1(s,Y^2_s,Z^2_s,V^2_s)-f^2(s,Y^2_s,Z^2_s,V^2_s)\\
		\varpi_s:&=\frac{f^1(s,Y^1_s,Z^1_s,V^1_s)-f^1(s,Y^2_s,Z^1_s,V^1_s)}{\overline{Y}_{s}}\mathds{1}_{\{\overline{Y}_{s}\neq 0\}}\\
		\beta_s:&=\frac{f^1(s,Y^2_s,Z^1_s,V^1_s)-f^1(s,Y^2_s,Z^2_s,V^1_s)}{\overline{Z}_{s}}\mathds{1}_{\{\overline{Z}_{s}\neq 0\}}
	\end{split}
	\right. 
\end{equation*}
Note that from (\ref{Mono}), we have
\begin{equation*}\label{Forza Inter}
	\begin{split}
		g(s)&=f^1(s,Y^2_s,Z^2_s,V^1_s)-f^1(s,Y^2_s,Z^2_s,V^2_s)+\varpi_s \overline{Y}_s +\beta_s \overline{Z}_s+f_s\\ 
		&\leq f_s+\varpi_s \overline{Y}_s +\beta_s \overline{Z}_s+ \sum_{k=1}^d \zeta^{Y^2_s,Z^2_s,V^1_s,V^2_s;(k)}_s \overline{V}^{(k)}_s {\lambda^{(k)}_s}.
	\end{split}
\end{equation*}
Let us consider the process $(\mathscr{E}_{t})_{t \in [0,T]}$ as the solution to the following forward SDE:
$$
d\mathscr{E}_{t} = \mathscr{E}_{t-} \left( \beta_t \, dB_t + \sum_{k=1}^d \zeta^{Y^2_s,Z^2_s,V^1_s,V^2_s;(k)}_t \, d\widetilde{N}^{(k)}_t \right), \quad \mathscr{E}_{0} = 1,
$$
where the process $\zeta$ is the one given in assumption \textbf{(H5)}. From assumption \textbf{(H5)}, we deduce that the random variables $\int_{0}^{T} \beta^2_s \, ds$ and $\sum_{k=1}^d \int_{0}^{T} \big|\zeta^{Y^2_s,Z^2_s,V^1_s,V^2_s;(k)}_s\big|^2 \lambda_s^{(k)} \, ds$ are bounded. Therefore, the stochastic exponential $\mathscr{E}$ is a square-integrable martingale (see, e.g., \cite[Ch.~II, Sec.~2.5]{Delong2013}). Moreover, using \textbf{(H5)}, we also have $\mathscr{E} \in \mathcal{S}^2_\beta$. By applying the compact formula for the stochastic exponential and using property \eqref{Mono-}, we obtain that $\mathscr{E}_t > 0$ for all $t \in [0,T]$ (see, e.g., Theorem 37 in \cite{Protter2005}, p.~84). Thus, the process $\mathscr{E}$ defines a martingale density, and we can introduce a new probability measure $\mathbb{Q}$ on $(\Omega, \mathcal{F})$, equivalent to $\mathbb{P}$, defined as follows:
$$
d \mathbb{Q}:=\mathscr{E}_T d\mathbb{P}
$$
Using Girsanov’s theorem (see, e.g., Theorem 36 in \cite{Protter2005}, p.~133, or \cite[Theorem 2.5.1]{Delong2013}), we deduce that
$$
\overline{B}_t := B_t - \int_{0}^{t} \beta_s \, ds \quad \text{and} \quad \overline{N}^{(k)}_t := \widetilde{N}^{(k)}_t - \int_{0}^{t} \zeta^{Y^2_s,Z^2_s,V^2_s,V^1_s;(k)}_t \lambda^{(k)}_s \, ds, \quad t \in [0,T], \quad k \in \{1, \ldots, d\}
$$
are, respectively, a Brownian motion and a compensated inhomogeneous Poisson process under the probability measure $\mathbb{Q}$ and the filtration $\mathbb{F}$.

Next, we apply Tanaka’s formula (Theorem 68 in \cite{Protter2005}, p.~213) to $\overline{Y}$, given by \eqref{Gs}, using the convex function $x \mapsto |x^+|^2$. Then, applying an integration by parts formula under $\mathbb{Q}$ to the process $e^{\beta A_t} |\overline{Y}^+_t|^2$, and taking into account the inequality $\varpi_s \leq \kappa_s \leq \alpha^2_s$ from assumption \textbf{(H2)}-(ii), we deduce that
\begin{equation}\label{RHS}
	\begin{split}
		&e^{\beta A_t}|\overline{Y}^+_t|^2+\beta \int_{t}^{T}e^{\beta A_t}|\overline{Y}^+_s|^2 dA_s\\
		& \leq 2\int_{t}^{T}e^{\beta A_s}\overline{Y}^+_s\left(f^2(s,Y^2_s,Z^2_s,V^2_s)-f^2(s,Y^2_s,Z^2_s,V^1_s)+\varpi_s \overline{Y}_s +\beta_s \overline{Z}_s+f_s\right) ds+2\int_{t}^{T}\overline{Y}^+_{s-} d\bar{K}^+_s\\
		&\quad-2\int_{t}^{T}e^{\beta A_s}\overline{Y}^+_{s-} d\bar{K}^-_s-2\int_{t}^{T}e^{\beta A_s}\overline{Y}^+_{s-} \overline{Z}_s dB_s-2\sum_{k=1}^d \int_{t}^{T} \overline{Y}_{s-}\overline{V}^{(k)}_s d \widetilde{N}^{(k)}_s\\
		& \leq 2\int_{t}^{T} e^{\beta A_s}\overline{Y}^+_{s-} \left(f_sds+ \overline{Y}_{s-} dA_s \right)
		-2\int_{t}^{T}e^{\beta A_s}\overline{Y}^+_{s-} \overline{Z}_s d\overline{B}_s-2\sum_{k=1}^d \int_{t}^{T}e^{\beta A_s} \overline{Y}^+_{s-}\overline{V}^{(k)}_s d \overline{N}^{(k)}_s\\
		&\quad+2\int_{t}^{T}e^{\beta A_s}\overline{Y}^+_{s-} d\bar{K}^+_s-2\int_{t}^{T}e^{\beta A_s}\overline{Y}^+_{s-} d\bar{K}^-_s.
	\end{split}
\end{equation}
Now, from the Skorokhod condition and our assumptions on $L^1$ and $L^2$, we observe that on the set $\Theta^+ := \{s \in [t,T] : Y^1_{s-} > Y^2_{s-} \}$, it holds that $Y^1_{s-} > Y^2_{s-} \geq L^2_{s-} \geq L^1_{s-}$. Therefore, on $\Theta^+$, we have $dK^{+,1}_s = 0$. Hence, we obtain
\begin{equation*}
	\begin{split}
		\int_{t}^{T}e^{\beta A_s}\overline{Y}^+_{s-}d\bar{K}^+_s&=\int_{t}^{T}e^{\beta A_s}\left({Y}^2_{s-}-{Y}^1_{s-}\right)\mathds{1}_{\{s \in \Theta^+\}}dK^{+,2}_s\\
		&=\int_{t}^{T}e^{\beta A_s}\left({L}^2_{s-}-{Y}^1_{s-}\right)\mathds{1}_{\{s \in \Theta^+\}}dK^{+,2}_s \leq 0.
	\end{split}
\end{equation*}
Similarly, we can prove that
$$
\int_{t}^{T}e^{\beta A_s}\overline{Y}^+_{s-}d\bar{K}^-_s\geq 0.
$$ 
Now, let $\{\tau_n\}_{n \geq 1}$ be a fundamental sequence of stopping times for the local martingales appearing on the right-hand side of \eqref{RHS}. For instance, we can define
$$
\tau_n := \inf\left\{ u \geq t : \int_{t}^{u} |\overline{Y}_s|^2 \left(  |\overline{Z}_s|^2 + \sum_{k=1}^d  |\overline{V}^{(k)}_s|^2 \big(1 + \zeta^{Y^2_s,Z^2_s,V^2_s,V^1_s;(k)}_s \big) \lambda^{(k)}_s \right) ds \geq n \right\} \wedge T, \quad \mathbb{Q}\text{-a.s.}
$$
Next, taking the conditional expectation with respect to $\mathcal{F}_t$ under $\mathbb{Q}$ on both sides of \eqref{RHS}, and using the assumption that $f_s \leq 0$, as well as choosing $\beta > 2$, we obtain
$$
e^{\beta A_t} |\overline{Y}^+_t|^2 \leq 0, \quad \mathbb{Q}\text{-a.s.}
$$
Hence, we conclude that $\overline{Y}^+_t = 0$ a.s. for each $t \in [0,T]$. Since $Y^1$ and $Y^2$ are RCLL processes, this implies that $Y^1 \leq Y^2$ $\mathbb{Q}$-a.s. on $[0,T]$, and thus also $\mathbb{P}$-a.s. on $[0,T]$, completing the proof.
\end{proof}

From Theorem \ref{Compa}, we derive the following remark:
\begin{remark}\label{Mono-prop}
\begin{itemize}
	\item If $L \equiv -\infty$, then $dK^{+} = 0$, and the comparison theorem also applies to upper reflected BSDEs.
	\item If $U \equiv +\infty$, then $dK^{-} = 0$, and the comparison theorem also applies to lower reflected BSDEs.
	\item If $L \equiv -\infty$ and $U \equiv +\infty$, then the comparison theorem reduces to the case of standard BSDEs.
\end{itemize}
\end{remark}

\section{Applications}
\label{sec44}
In this section, we present two closely related applications of the preceding results. The first concerns the nonlinear valuation of an American game option in a financial market driven by an inhomogeneous Lévy process. The second is devoted to the formulation and analysis of a generalized Dynkin game under nonlinear expectation.

\subsection{American game option in a Lévy market under the non-linear $\mathcal{E}^f$-expectation}
\label{sec4}
In this section, we show that the doubly reflected BSDE introduced in Section \ref{DRBSDE} can be used to solve the problem of pricing American game options in a market model driven by a non-homogeneous Lévy process, as studied in \cite{ElJamaliOtmani2022}. We consider a financial model that extends the classical Black–Scholes framework by incorporating jumps, described by the following forward stochastic differential equations:
\begin{equation}\label{Market}
	\left\lbrace
	\begin{split}
		dS^0_t &= r_t S^0_t \, dt, \quad S^0_0 = 1;\\
		dS_t &= S_{t-} \, dX_t, \quad S_0 > 0,
	\end{split} 
	\right.
\end{equation}
where $(S^0_t)_{t \leq T}$ represents the price process of a risk-free asset with interest rate $(r_t)_{t \leq T}$, and $(S_t)_{t \leq T}$ denotes the price of a risky asset. As shown in the Lévy market setting (see \cite[Section 5]{Elotmani2009} and the detailed discussion in \cite[p.~211]{ElJamaliOtmani2022}), the market model \eqref{Market} can be completed. Consequently, we treat \eqref{Market} as a complete market model in the sequel.

We assume that the barriers $(L, U)$ in the DRBSDE \eqref{basic equation} are given in the following Markovian form:
\begin{equation}\label{Markov}
	L_t = L(S_t), \quad \text{and} \quad U_t = U(S_t),
\end{equation}
where $L(\cdot)$ and $U(\cdot)$ are measurable functions $L, U: \mathbb{R} \rightarrow \mathbb{R}$ that are jointly continuous. Moreover, there exist constants $\mathfrak{c} > 0$ and $p \geq 1$ such that
$$
|L(x)| + |U(x)| \leq \mathfrak{c} \left(1 + |x|^p\right), \quad \forall x \in \mathbb{R}.
$$

To simplify the analysis and adapt assumptions \textbf{(H3)}-(i) and \textbf{(H3)}-(iii) to this Markovian setting, we consider the following modification:
\begin{itemize}
	\item[\textbf{(H3')}] We assume that:
	\begin{itemize}
		\item[(i)] $L(S_T) = U(S_T)$, $\mathbb{P}$-a.s.;
		\item[(iii)] $L(S_t) < U(S_t)$ and $L(S_{t-}) < U(S_{t-})$ for all $t \in [0,T[$, $\mathbb{P}$-a.s.
	\end{itemize}
\end{itemize}

In what follows, we refer to assumption \textbf{(H3')} as assumption \textbf{(H3)}, replacing \textbf{(H3)}-(i) and \textbf{(H3)}-(iii) with \textbf{(H3')}-(i) and \textbf{(H3')}-(iii), while retaining the integrability condition \textbf{(H3)}-(ii). This latter condition holds, for instance, when $\int_{0}^{T} |c_s| \, ds$ and $\sum_{k=1}^d \int_0^T \big|\gamma_s^{(k)}\big|^p \lambda_s^{(k)} \, ds$ are bounded for any $p \geq 2$ (see \cite[Proposition A.1]{Quenez2013}), together with assumption \textbf{(H5)}.

We now consider an American game option, which is a contract between a buyer and a seller, where both parties are allowed to exercise their respective rights at any stopping time before maturity $T$. If the seller chooses a cancellation time $\tau_2 \in \mathcal{T}_{0}^T$ and the buyer chooses an exercise time $\tau_1 \in \mathcal{T}_{0}^T$, then the seller pays the buyer the following payoff at time $\tau_1 \wedge \tau_2$:
\begin{equation}\label{Our-Payoff}
	\mathcal{P}(\tau_1,\tau_2) = L(S_{\tau_1}) \mathds{1}_{\{\tau_1 \leq \tau_2\}} + U(S_{\tau_2}) \mathds{1}_{\{\tau_2 < \tau_1\}}.
\end{equation}

The \textit{nonlinear pricing system}, also referred to as $f$-evaluation in the terminology of Peng \cite{Peng1999,Peng2004}, and denoted by $\mathcal{E}^f$, generalizes the classical linear expectation in order to extend the concept of expected utility (see, e.g., \cite{Duffie1992}), a foundational principle in modern economic theory, rooted in the work of Von Neumann and Morgenstern. To define this notion, we consider the following backward stochastic differential equation (BSDE) with terminal time $T$, terminal condition $\xi$, and driver $f$:
\begin{equation}\label{BSDEs}
	Y_{t} = \xi + \int_{t}^{T} f(s, Y_s, Z_s, V_s) \, ds
	- \int_{t}^{T} Z_s \, dB_s - \sum_{k=1}^d \int_{t}^{T} V^{(k)}_s \, d\widetilde{N}^{(k)}_s, \quad t \in [0,T].
\end{equation}

We now formally define the $f$-evaluation operator in this context:

\begin{definition}
	Let $\tau \in \mathcal{T}_{0}^T$ and let $\xi \in \mathbb{L}^2_{\beta}(\mathcal{F}_\tau)$. The $f$-evaluation of $\xi$ at time $t$ with horizon $\tau$ is defined by $\mathcal{E}^f_{t,\tau}(\xi) := X_t$, for all $t \in [0,T]$, $\mathbb{P}$-a.s., where $X$ is the first component of the solution to the BSDE \eqref{BSDEs} with driver $f \mathds{1}_{\{t \leq \tau\}}$, terminal time $T$, and terminal condition $\xi$. More generally, we define the operator $\mathcal{E}^f: (\tau, \xi) \mapsto \mathcal{E}^f_{\cdot, \tau}(\xi)$, and refer to $\mathcal{E}^f_{\cdot, \tau}(\xi)$ as the $\mathcal{E}^f$-expectation process of $\xi$.
\end{definition}

In this framework, the upper and lower values of the game option, denoted respectively by $\overline{\mathbf{V}}$ and $\underline{\mathbf{V}}$, are expressed via nonlinear expectations as follows:
$$
\overline{\mathbf{V}} = \inf_{\tau_2 \in \mathcal{T}_{0}^T} \sup_{\tau_1 \in \mathcal{T}_{0}^T} \mathcal{E}^f_{0,\tau_1 \wedge \tau_2}\big(\mathcal{P}(\tau_1,\tau_2)\big), \quad
\underline{\mathbf{V}} = \sup_{\tau_1 \in \mathcal{T}_{0}^T} \inf_{\tau_2 \in \mathcal{T}_{0}^T} \mathcal{E}^f_{0,\tau_1 \wedge \tau_2}\big(\mathcal{P}(\tau_1,\tau_2)\big).
$$

Since American game options can be formulated as Dynkin games in financial mathematics, the pricing problem naturally leads to the valuation of a generalized Dynkin game—a formulation first introduced by Dumitrescu et al.~\cite{D2016} and further studied in a defaultable setting in \cite{DumitrescuQuenezSulem2017}. In this framework, we say that the game option admits an $\mathcal{E}^f$-value $\mathbf{V}$ if the upper and lower values coincide:
$$
\underline{\mathbf{V}} = \mathbf{V} = \overline{\mathbf{V}}.
$$

In order to solve the pricing problem of this game option in the defaultable market model (\ref{Market}), we consider the DRBSDE with two reflecting barriers given as follows:
\begin{equation}
	\left\{
	\begin{split}
		\text{(i)} &~ \mathbb{P}\text{-a.s. for all } t \in [0,T]\\
		&~ Y_{t}=L(S_T)+\int_{t}^{T} f(s,Y_s,Z_s,V_s)ds+\left( K^{+}_{ T}-K^{+}_{t}\right)-\left( K^{-}_{ T}-K^{-}_{t }\right)\\
		&\qquad\qquad- \int_{t}^{T} Z_s dB_s - \sum_{k=1}^d \int_{t}^{T} V^{(k)}_s  d\widetilde{N}^{(k)}_s;\\
		\text{(ii)} &~  L(S_t) \leq Y_t \leq U(S_t),~  \forall t \in [0,T],~\mathbb{P}\text{-a.s.};\\
		\text{(iii)} &~ \int_0^{T  }(Y_{t-}-L(S_{t-}))dK^{+}_t=\int_0^{T  }(U(S_{t-})-Y_{t-})dK^{-}_t=0.
	\end{split}
	\right.
	\label{basic equation Finance}
\end{equation}
To ensure the well-posedness of the DRBSDE \eqref{basic equation Finance} (see Theorem \ref{basic Thm}) and the monotonicity of the nonlinear expectation (see Theorem \ref{Compa} and Remark \ref{Mono-prop}), which will be used in the sequel, we assume that assumptions \textbf{(H1)}, \textbf{(H2)}, \textbf{(H3')}, \textbf{(H4)}, and \textbf{(H5)} hold true.

The following definitions will be used in the subsequent discussion:
\begin{definition}[Strong $\mathcal{E}^f$-(sub,super)martingale]
	Let $\beta >0$, $Y \in \mathcal{S}^2_{\beta}$, and $\sigma,\tau \in \mathcal{T}_{0}^T$ such that $\sigma \leq \tau$ a.s. The process $Y$ is said to be a strong $\mathcal{E}^f$-supermartingale (resp. $\mathcal{E}^f$-submartingale) on $[\sigma,\tau]$, if for all $\sigma^{\star}, \tau^{\star} \in \mathcal{T}_{0}^T$ such that $\sigma \leq \sigma^{\star} \leq \tau^{\star} \leq \tau$ a.s., we have $\mathcal{E}^f_{\sigma^{\ast},\tau^{\ast}}(Y_{\eta^{\ast}}) \leq Y_{\sigma^{\ast}}$ (resp. $\mathcal{E}^f_{\sigma^{\ast},\tau^{\ast}}(Y_{\tau^{\ast}}) \geq Y_{\sigma^{\ast}}$) a.s. Finally, the process $Y$ is said to be a strong $\mathcal{E}^f$-martingale on $[\sigma,\tau]$, if it is both a strong $\mathcal{E}^f$-supermartingale and a strong $\mathcal{E}^f$-submartingale on $[\sigma,\tau]$.
	\label{definition strong on}
\end{definition}

\begin{definition}[$\theta$-saddle point for the game option]
	Let $\theta \in \mathcal{T}_{0}^T$. A pair $\left(\tau_1^{\star},\tau^{\star}_2\right)\in \mathcal{T}_{\theta}^T \times \mathcal{T}_{\theta}^T$ is called a $\theta$-saddle point for the
	game option if, for each $\left(\tau_1,\tau_2\right)\in \mathcal{T}_{\theta}^T \times \mathcal{T}_{\theta}^T$, we have
	$$
	\mathcal{E}^f_{\theta,\tau_1 \wedge \tau^{\star}_2}\left(\mathcal{P}(\tau_1,\tau^{\star}_2)\right) \leq  \mathcal{E}^f_{\theta,\tau^{\star}_1 \wedge \tau^{\star}_2}\left(\mathcal{P}(\tau^{\star}_1,\tau^{\star}_2)\right) \leq  \mathcal{E}^f_{\theta,\tau^{\star}_1 \wedge \eta_2}\left(\mathcal{P}(\tau^{\star}_1,\tau_2)\right).
	$$
	\label{Sadlle point definition}
\end{definition}
\begin{definition}[Left-upper semi-continuity]
	An optional process $\left( \phi_t\right)_{t \leq T}$ is said to be left-upper semi-continuous (l.u.s.c., for short) along stopping times if, for any stopping time $\tau \in \mathcal{T}_{0}^T$ and any non-decreasing sequence of stopping times $\left\{\tau_n\right\}_{n \in \mathbb{N}} \in  \left( \mathcal{T}_{0}^T\right)^{\mathbb{N}}$ such that $\tau^n \uparrow \tau$ a.s., we have
	$
	\limsup_{n \rightarrow +\infty} \phi_{\tau_n} \leq \phi_{\tau}$ a.s.
\end{definition}
\begin{remark}
	It is easy to verify that, in the particular case of an optional process $(\phi_t)_{t \leq T}$ with finite left limits, the process $(\phi_t)_{t \leq T}$ is left-upper semicontinuous (l.u.s.c.) along stopping times if and only if, for every predictable stopping time $\tau \in \mathcal{T}_{0}^T$, we have $\Delta \phi_{\tau} \geq 0$ a.s. (see Proposition 5.1 in \cite{ElmansouriOtmani2024}).
	\label{Remark continuity}
\end{remark}
From Remarks \ref{Reflection remark} and \ref{Remark continuity}, we may deduce the following proposition:
\begin{proposition}
	Assume that the DRBSDE (\ref{basic equation Finance}) admits a unique solution $\left(Y_t,Z_t,V_t,K^{+}_t,K^{-}_t\right)_{t \leq T}$. If $L\left(S_{\cdot}\right)$ (resp. $-U\left(S_{\cdot}\right)$) is l.u.s.c. along stopping times, then the reflection process $K^{+}$ (resp. $K^{-}$) is continuous on $[0,T]$.
	\label{Continuity of reflection processes}
\end{proposition}
The principal result of the current section can now be stated as follows:
\begin{theorem}
	Let $\big(Y_t,Z_t,V_t,K^{+}_t,K^{-}_t\big)_{t \leq T}$ be the unique solution of the DRBSDE (\ref{basic equation Finance}). For any $\theta \in \mathcal{T}_{0}^T$, the $\mathcal{E}^f$-value of our game option at any time $\theta$ is given by:
	\begin{equation*}
		\essinf_{\eta_2 \in \mathcal{T}_{\theta}^T} \esssup_{\eta_1 \in \mathcal{T}_{\theta}^T} \mathcal{E}^f_{\theta,\eta_1 \wedge \eta_2}(\mathcal{P}(\eta_1,\eta_2))=Y_{\theta}=\esssup_{\eta_1 \in \mathcal{T}_{\theta}^T} \essinf_{\eta_2 \in \mathcal{T}_{\theta}^T} \mathcal{E}^f_{\theta,\eta_1 \wedge \eta_2}(\mathcal{P}(\eta_1,\eta_2)).
		\label{Wanted inequality}
	\end{equation*}
	In particular $\mathbf{V}=Y_0$, i.e. the initial value $Y_0$ represents the $\mathcal{E}^f$-value of the game option.
	
	Assume moreover, that $L(S_{\cdot})$ and $-U(S_{\cdot})$ are l.u.s.c.  along stopping times, then the pair $(	\eta^{\theta,\star}_1,	\eta^{\theta,\star}_2)$ given by
	$$
	\eta^{\theta,\star}_1=\inf\{t \geq \theta:Y_t=L(S_t)\} \text{ and } \eta^{\theta,\star}_2=\inf\{t \geq \theta:Y_t=U(S_t)\},
	$$
	defines a $\theta$-saddle point for the game option.
	\label{Main result}
\end{theorem}
\begin{proof}
	\textbf{Part 1: the state process $Y$ is the $\mathcal{E}^f$-value of the game option}\\  
	
	The procedure is comprised of five distinct steps, each of which incorporates an assessment that must be demonstrated.
	\begin{enumerate}
		\item For each $\varepsilon>0$, consider the following two stopping times $\eta^{\theta,\varepsilon}_1$ and $\eta^{\theta,\varepsilon}_2$ defined as
		$$
		\eta^{\theta,\varepsilon}_1=\inf\{t \geq \theta:Y_t \leq L(S_t)+\varepsilon\}  \text{ and } \eta^{\theta}_2=\inf\{t \geq \theta:Y_t \geq U(S_t)-\varepsilon\}.
		$$
		The reflection process $K^{+}$ is a.s. constant on $\big[\theta,\eta^{\theta,\varepsilon}_1\big]$  and $K^{-}$ is a.s. constant on $\big[\theta,\eta^{\theta,\varepsilon}_2\big]$.\\
		Firstly, note that $\eta^{\theta,\varepsilon}_1$ and $\eta^{\theta,\varepsilon}_2$ are valued in $[0,T]$ as  $Y_T=L\big(S_T\big)=U\big(S_T\big)$ a.s.\\ 
		Next, for almost every $\omega \in \Omega$, based on the definition of the stopping time $\eta^{\theta,\varepsilon}_1$ and Remark \ref{Reflection remark}, we have $Y_t(\omega)>L\big(S_t(\omega)\big)+\varepsilon>L\big(S_t(\omega)\big)$ for all $t \in \big[\theta(\omega),\eta^{\theta,\varepsilon}_1(\omega)\big[$. Hence, for almost every $\omega \in \Omega$, the continuous function $t \mapsto K^{+,c}_t(\omega)$ is constant on $\big[ \theta(\omega),\eta^{\theta,\varepsilon}_1(\omega)\big[$. By continuity, we deduce that $t \mapsto K^{+,c}_t(\omega)$ is constant on $\big[ \theta(\omega),\eta^{\theta,\varepsilon}_1(\omega)\big]$. On the other hand, it is clear that for almost every $\omega \in \Omega$, we have $Y_{t-} \geq L\big(S_{t-}\big)+\varepsilon$ for all $t \in \big[\theta(\omega),\eta^{\theta,\varepsilon}_1(\omega)\big[$. Therefore, the purely discontinuous function $t \mapsto K^{+,d}_t(\omega)$ is constant on $\big[ \theta(\omega),\eta^{\theta,\varepsilon}_1(\omega)\big[$. Moreover, as $Y_{\eta^{\theta,\varepsilon}_1-} \geq L\big(S_{\eta^{\theta,\varepsilon}_1-}\big)+\varepsilon>L\big(S_{\eta^{\theta,\varepsilon}_1-}\big) $ a.s., we deduce that $\Delta K^{+,d}_{\eta^{\theta,\varepsilon}_1}=0$. Then, for almost every $\omega \in \Omega$, the function $t \mapsto K^{+,d}_t(\omega)$ is constant on $\big[ \theta(\omega),\eta^{\theta,\varepsilon}_1(\omega)\big]$. Hence, almost surely, $K^{+}$ is constant on $\big[ \theta,\eta^{\theta,\varepsilon}_1\big]$ as $K^{+}=K^{+,c}+K^{+,d}$.\\ 
		Similarly, $K^{-}$ is almost surely constant on $\big[ \theta,\eta^{\theta,\varepsilon}_2\big]$.
		\item Let us verify that the state process $Y$ is a strong $\mathcal{E}^f$-submartingale on $[\theta,\eta^{\theta,\varepsilon}_1]$ and a strong $\mathcal{E}^f$-supermartingale on $[\theta,\eta^{\theta,\varepsilon}_2]$.\\
		From the previous step, we deduce that $(Y_t,Z_t,V_t,K^{-}_t)_{t \leq T}$ is a solution on $\big[ \theta,\eta^{\theta,\varepsilon}_1\big]$ of the reflected BSDE associated with driver $f$, upper obstacle $U\big(S_{\cdot}\big)$, terminal time $\eta^{\theta,\varepsilon}_1$, and a terminal condition $Y_{\eta^{\theta,\varepsilon}_1}$. Similarly, the process $(Y_t,Z_t,V_t,K^{+}_t)_{t \leq T}$ is a solution on $\big[ \theta,\eta^{\theta,\varepsilon}_2\big]$ of the reflected BSDE associated with driver $f$, lower barrier $L\big( S^1_{\cdot}\big)$, terminal time $\eta^{\theta,\varepsilon}_2$, and a terminal condition $Y_{\eta^{\theta,\varepsilon}_2}$. Then, from Remark \ref{Mono-prop}, we deduce that $Y_t \leq \mathcal{E}^f_{t,\eta^{\theta,\varepsilon}_1}\big(Y_{\eta^{\theta,\varepsilon}_1}\big)$ for all $t \in \big[\theta,\eta^{\theta,\varepsilon}_1\big]$ a.s. and $Y_t \geq \mathcal{E}^f_{t,\eta^{\theta,\varepsilon}_2}\big(Y_{\eta^{\theta,\varepsilon}_2}\big)$ for all $t \in \big[\theta,\eta^{\theta,\varepsilon}_2\big]$ a.s.
		\item Let us show that 
		$$Y_{\eta^{\theta,\varepsilon}_1}\leq L\left(S_{\eta^{\theta,\varepsilon}_1}\right)+\varepsilon\ a.s. \mbox{ and }  U\left(S_{\eta^{\theta,\varepsilon}_2}\right)-\varepsilon \leq Y_{\eta^{\theta,\varepsilon}_2}.$$
		Let us fix $\omega \in \Omega$. As $T \in \{t \geq \theta:Y_t(\omega) \leq L(S_t(\omega))+\varepsilon\} \cap \{t \geq \theta:Y_t(\omega) \geq U(S_t(\omega))-\varepsilon\}$,  then by a basic property of the infimum, there exists two sequences $\{\mathfrak{t}^1_n(\omega)\}_{n \in \mathbb{N}}$ and $\{\mathfrak{t}^2_n(\omega)\}_{n \in \mathbb{N}}$ such that $\mathfrak{t}^1_n(\omega) \downarrow \eta^{\theta,\varepsilon}_1(\omega)$, $\mathfrak{t}^2_n(\omega) \downarrow \eta^{\theta,\varepsilon}_2(\omega)$ as $n \to +\infty$ and 
		$$Y_{\mathfrak{t}^1_n(\omega)} \leq L\left(S_{\mathfrak{t}^1_n(\omega)}\right)+\varepsilon,\  Y_{\mathfrak{t}^2_n(\omega)} \geq U\left(S_{\mathfrak{t}^2_n(\omega)}\right)-\varepsilon.$$ 
		The result follows then from  the right-continuity of $Y$, $L(S_{\cdot})$ and $U(S_{\cdot})$.
		\item We propose to show that, for any $\eta \in \mathfrak{T}_{\theta,T}$, 
		$$\mathcal{E}^f_{\theta,\eta \wedge\eta^{\theta,\varepsilon}_2}\left(\mathcal{P}(\eta,\eta^{\theta,\varepsilon}_2) \right)-\mathfrak{C}_{\beta} \varepsilon \leq Y_{\theta}\leq \mathcal{E}^f_{\theta,\eta^{\theta,\varepsilon}_1 \wedge \eta }\left(\mathcal{P}(\eta^{\theta,\varepsilon}_1,\eta) \right)+\mathfrak{C}_{\beta} \varepsilon.$$
		Indeed, let $\eta \in \mathfrak{T}_{\theta,T}$.
		From the second step, we conclude that $Y_{\theta} \geq \mathcal{E}^f_{\theta,\eta \wedge\eta^{\theta,\varepsilon}_2}\big(Y_{\eta \wedge\eta^{\theta,\varepsilon}_2}\big)$.\\ 
		Now, on the one hand, from DRBSDE \eqref{basic equation Finance}-(ii), we have $Y \geq L(S_{\cdot})$ on $[0,T]$. On the other hand, using the third step of the current proof, we derive $U\big(S_{\eta^{\theta,\varepsilon}_2}\big)-\varepsilon \leq Y_{\eta^{\theta,\varepsilon}_2}$ a.s.. Hence, combining these two facts, we deduce that
		$$Y_{\eta \wedge\eta^{\theta,\varepsilon}_2} \geq L\left(S_{\eta}\right) \mathds{1}_{\{\eta \leq \eta^{\theta,\varepsilon}_2\}}+\left(U\left(S_{\eta^{\theta,\varepsilon}_2}\right)-\varepsilon\right)\mathds{1}_{\{ \eta^{\theta,\varepsilon}_2<\eta\}}  \geq \mathcal{P}(\eta,\eta^{\theta,\varepsilon}_2)-\varepsilon.$$
		Then, we get 
		$$Y_{\theta} \geq \mathcal{E}^f_{\theta,\eta \wedge\eta^{\theta,\varepsilon}_2}\left(Y_{\eta \wedge\eta^{\theta,\varepsilon}_2}\right) \geq \mathcal{E}^f_{\theta,\eta \wedge\eta^{\theta,\varepsilon}_2}\left(\mathcal{P}(\eta,\eta^{\theta,\varepsilon}_2)-\varepsilon \right) .$$
		
		Now, using assumption \textbf{(H5)} and Proposition \ref{Useful proposition}, we deduce that
		$$
		\left| \mathcal{E}^f_{\theta,\eta \wedge\eta^{\theta,\varepsilon}_2}\left(\mathcal{P}(\eta,\eta^{\theta,\varepsilon}_2)-\varepsilon \right)-\mathcal{E}^f_{\theta,\eta \wedge\eta^{\theta,\varepsilon}_2}\left(\mathcal{P}(\eta,\eta^{\theta,\varepsilon}_2) \right) \right| \leq \mathfrak{C}_{\beta} \varepsilon,
		$$
		In particular, we have
		$$\mathcal{E}^f_{\theta,\eta \wedge\eta^{\theta,\varepsilon}_2}\left(\mathcal{P}(\eta,\eta^{\theta,\varepsilon}_2)-\varepsilon \right)\geq \mathcal{E}^f_{\theta,\eta \wedge\eta^{\theta,\varepsilon}_2}\left(\mathcal{P}(\eta,\eta^{\theta,\varepsilon}_2) \right)-\mathfrak{C}_{\beta} \varepsilon. $$
		Therefore
		$$Y_{\theta}\geq \mathcal{E}^f_{\theta,\eta \wedge\eta^{\theta,\varepsilon}_2}\left(\mathcal{P}(\eta,\eta^{\theta,\varepsilon}_2) \right)-\mathfrak{C}_{\beta} \varepsilon.$$
		In a similar way, we can show that $Y_{\theta}\leq \mathcal{E}^f_{\theta,\eta \wedge\eta^{\theta,\varepsilon}_1}\big(\mathcal{P}(\eta^{\theta,\varepsilon}_1,\eta) \big)+\mathfrak{C}_{\beta} \varepsilon$.
		\item From the forth step, we can easily derive that
		$$
		Y_{\theta}\geq\esssup_{\eta_1 \in \mathcal{T}_{\theta}^T}\mathcal{E}^f_{\theta,\eta_1 \wedge\eta^{\theta,\varepsilon}_2}\left(\mathcal{P}(\eta_1,\eta^{\theta,\varepsilon}_2) \right)-\mathfrak{C}_{\beta} \varepsilon,
		$$
		and
		$$
		Y_{\theta}\leq \essinf_{\eta_2 \in \mathcal{T}_{\theta}^T} \mathcal{E}^f_{\theta,\eta^{\theta,\varepsilon}_1 \wedge \eta_2 }\left(\mathcal{P}(\eta^{\theta,\varepsilon}_1,\eta_2) \right)+\mathfrak{C}_{\beta} \varepsilon.
		$$
		Then
		$$
		Y_{\theta} \geq \essinf_{\eta_2 \in \mathcal{T}_{\theta}^T}\esssup_{\eta_1 \in \mathcal{T}_{\theta}^T}\mathcal{E}^f_{\theta,\eta_1 \wedge\eta_2}\left(\mathcal{P}(\eta_1,\eta_2) \right)-\mathfrak{C}_{\beta} \varepsilon,
		$$
		and
		$$
		Y_{\theta}\leq \esssup_{\eta_1 \in \mathcal{T}_{\theta}^T}\essinf_{\eta_2 \in \mathcal{T}_{\theta}^T} \mathcal{E}^f_{\theta,\eta_1 \wedge \eta_2 }\left(\mathcal{P}(\eta_1,\eta_2) \right)+\mathfrak{C}_{\beta} \varepsilon.
		$$
		As the above two inequalities hold for almost every $\omega \in \Omega$ and for any $\varepsilon>0$, passing to the limits on both sides together with the fact that
		$$
		\esssup_{\eta_1 \in \mathcal{T}_{\theta}^T}\essinf_{\eta_2 \in \mathcal{T}_{\theta}^T} \mathcal{E}^f_{\theta,\eta_1 \wedge \eta_2 }\left(\mathcal{P}(\eta_1,\eta_2) \right)\leq \essinf_{\eta_2 \in \mathcal{T}_{\theta}^T}\esssup_{\eta_1 \in \mathcal{T}_{\theta}^T}\mathcal{E}^f_{\theta,\eta_1 \wedge\eta_2}\left(\mathcal{P}(\eta_1,\eta_2) \right),
		$$
		we conclude that the state variable $Y$ of the DRBSDE \eqref{basic equation Finance} coincides with the $\mathcal{E}^f$-value process of the game option
	\end{enumerate}
	\textbf{ Part 2: The pair $(	\eta^{\theta,\star}_1,	\eta^{\theta,\star}_2)$ is a $\theta$-saddle point} \\
	
	From Proposition \ref{Continuity of reflection processes}, we deduce that the state process $Y$ is a strong $\mathcal{E}^f$-martingale on $[\theta,\eta^{\theta,\star}_1 \wedge \eta^{\theta,\star}_2]$. Moreover, since $Y$, $L(S_{\cdot})$ and $U(S_{\cdot})$  are right-continuous processes, we have $Y_{\eta^{\theta,\star}_1}=L\big(S_{\eta^{\theta,\star}_1}\big)$ and $Y_{\eta^{\theta,\star}_2}=U\big(S_{\eta^{\theta,\star}_2}\big)$. Therefore, we have:
	\begin{equation*}
		\begin{split}
			Y_{\theta}&=\mathcal{E}^f_{\theta,\eta^{\theta,\star}_1 \wedge \eta^{\theta,\star}_2}\left(L\big(S_{\eta^{\theta,\star}_1}\big)\mathds{1}_{\{\eta^{\theta,\star}_1 \leq \eta^{\theta,\star}_2\}}+U\big(S_{\eta^{\theta,\star}_2}\big) \mathds{1}_{\{\eta^{\theta,\star}_2<\eta^{\theta,\star}_1\}} \right)\\
			&=\mathcal{E}^f_{\theta,\eta^{\theta,\star}_1 \wedge \eta^{\theta,\star}_2}\left(\mathcal{P}(\eta^{\theta,\star}_1,\eta^{\theta,\star}_2)\right).
		\end{split}
	\end{equation*}
	Now, a similar argument as the one used in the second step of the Part 1, together with the continuity of the processes $K^{+}$, allows us to conclude that $Y$ is a strong $\mathcal{E}^f$-submartingale on $[\theta,\eta^{\theta,\star}_1]$. Henceforth, for any $\eta_2 \in \mathcal{T}_{\theta}^T$, we have
	$$
	Y_{\theta} \leq \mathcal{E}^f_{\theta, \eta^{\theta,\star}_1 \wedge \eta_2}\left(Y_{\eta^{\theta,\star}_1 \wedge \eta_2}\right)=\mathcal{E}^f_{\theta,\eta^{\theta,\star}_1 \wedge \eta_2}\left(L\big(S_{\eta^{\theta,\star}_1}\big)\mathds{1}_{\{\eta^{\theta,\star}_1 \leq \eta_2\}}+Y_{\eta_2} \mathds{1}_{\{\eta_2<\eta^{\theta,\star}_1\}} \right).
	$$
	Using again the monotonicity of the $\mathcal{E}^f$-expectation and $Y_{\eta_2}  \leq U_{\eta_2} $, we deduce that
	$$	Y_{\theta} \leq \mathcal{E}^f_{\theta,\eta^{\theta,\star}_1 \wedge \eta_2}\left(\mathcal{P}(\eta^{\theta,\star}_1,\eta_2)\right) \quad \text{ a.s.} $$
	Then 
	$$\mathcal{E}^f_{\theta,\eta^{\theta,\star}_1 \wedge \eta^{\theta,\star}_2}\left(\mathcal{P}(\eta^{\theta,\star}_1,\eta^{\theta,\star}_2)\right) \leq \mathcal{E}^f_{\theta,\eta^{\theta,\star}_1 \wedge \eta_2}\left(\mathcal{P}(\eta^{\theta,\star}_1,\eta_2)\right),\quad \text{ a.s.}$$ 
	
	Similarly, using the fact that $Y$ is a strong $\mathcal{E}^f$-supermartingale on $[\theta,\eta^{\theta,\star}_2]$, for any $\eta_1 \in \mathcal{T}_{\theta}^T$, we derive that $$
	\mathcal{E}^f_{\theta,\eta_1 \wedge \eta^{\star}_2}\left(\mathcal{P}(\eta_1,\eta^{\star}_2)\right) \leq  \mathcal{E}^f_{\theta,\eta^{\star}_1 \wedge \eta^{\star}_2}\left(\mathcal{P}(\eta^{\star}_1,\eta^{\star}_2)\right),\quad \text{ a.s.}
	$$ 
	
	Completing the proof.
\end{proof}

\subsection{Generalized Dynkin game}

Building on the results established in Section \ref{sec4}, we now turn to a game-theoretic application by considering generalized Dynkin games under nonlinear expectation.

For any $\theta \in \mathcal{T}_{0,T}$ and any pair $(\tau,\sigma) \in \mathcal{T}_{\theta,T}\times\mathcal{T}_{\theta,T}$, we define the payoff of the generalized Dynkin game by
\begin{equation}
	\mathcal{I}(\tau,\sigma)
	:=
	L_{\tau}\mathds{1}_{\{\tau\leq \sigma\}}
	+
	U_{\sigma}\mathds{1}_{\{\sigma<\tau\}}.
	\label{payoff generalized Dynkin game}
\end{equation}

The upper and lower values of this game is given respectively by:
\begin{equation}
	\underline{\mathcal{V}}_{\theta}
	:=
	\esssup_{\tau\in\mathcal{T}_{\theta,T}}
	\essinf_{\sigma\in\mathcal{T}_{\theta,T}}
	\mathcal{E}^{f}_{\theta,\tau\wedge\sigma}
	\left(\mathcal{I}(\tau,\sigma)\right),
	\label{lower value generalized Dynkin game}
\end{equation}
and
\begin{equation}
	\overline{\mathcal{V}}_{\theta}
	:=
	\essinf_{\sigma\in\mathcal{T}_{\theta,T}}
	\esssup_{\tau\in\mathcal{T}_{\theta,T}}
	\mathcal{E}^{f}_{\theta,\tau\wedge\sigma}
	\left(\mathcal{I}(\tau,\sigma)\right).
	\label{upper value generalized Dynkin game}
\end{equation}

The main result is given in the following theorem.
\begin{theorem}
	Let $(Y,Z,V,K^{+},K^{-})$ be the unique solution of the DRBSDE \eqref{basic equation} associated with the driver $f$, terminal condition $\xi$, lower obstacle $L$, and upper obstacle $U$. Then, for every $\theta\in\mathcal{T}_{0,T}$,
	\begin{equation}
		Y_{\theta}
		=
		\underline{\mathcal{V}}_{\theta}
		=
		\overline{\mathcal{V}}_{\theta},
		\qquad \text{a.s.}
		\label{value generalized Dynkin game}
	\end{equation}
	Consequently, the generalized Dynkin game \eqref{payoff generalized Dynkin game}--\eqref{upper value generalized Dynkin game} admits a value under the nonlinear expectation $\mathcal{E}^{f}$, and this value is characterized by the first component of the solution to the associated DRBSDE.
\end{theorem}

Let us consider a fixed stopping time $\theta\in\mathcal{T}_{0,T}$ and $\varepsilon>0$, and then define the stopping times
\begin{equation*}
	\tau_{\theta}^{\varepsilon}
	:=
	\inf\left\{t\geq\theta:Y_t\leq L_t+\varepsilon\right\}\wedge T,
	\qquad
	\sigma_{\theta}^{\varepsilon}
	:=
	\inf\left\{t\geq\theta:Y_t\geq U_t-\varepsilon\right\}\wedge T.
	\label{epsilon optimal stopping times}
\end{equation*}
Then, $\tau_{\theta}^{\varepsilon}$ and $\sigma_{\theta}^{\varepsilon}$ are approximately optimal stopping times for the maximizing and minimizing players, respectively. More precisely, there exists a constant $\mathfrak{C}>0$, independent of $\varepsilon$, such that the pair $\left(\tau_{\theta}^{\varepsilon},\sigma_{\theta}^{\varepsilon}\right)$ constitutes a $\mathfrak{C}\varepsilon$-saddle point for the generalized Dynkin \eqref{payoff generalized Dynkin game}--\eqref{upper value generalized Dynkin game}.

Furthermore, assume that the lower obstacle $L$ is left-upper semicontinuous along stopping times and that the upper obstacle $U$ is left-lower semicontinuous along stopping times. For every $\theta\in\mathcal{T}_{0,T}$, define
\begin{equation*}
	\tau_{\theta}^{\ast}
	:=
	\inf\left\{t\geq\theta:Y_t=L_t\right\}\wedge T,
	\qquad
	\sigma_{\theta}^{\ast}
	:=
	\inf\left\{t\geq\theta:Y_t=U_t\right\}\wedge T.
	\label{optimal stopping times generalized Dynkin game}
\end{equation*}
Then, the pair $\left(\tau_{\theta}^{\ast},\sigma_{\theta}^{\ast}\right)$ is a saddle point for the generalized Dynkin game \eqref{payoff generalized Dynkin game}--\eqref{upper value generalized Dynkin game}. In particular, for every $\tau,\sigma\in\mathcal{T}_{\theta,T}$,
\begin{equation*}
	\mathcal{E}^{f}_{\theta,\tau\wedge\sigma_{\theta}^{\ast}}
	\left(\mathcal{I}\left(\tau,\sigma_{\theta}^{\ast}\right)\right)
	\leq
	Y_{\theta}
	=
	\mathcal{E}^{f}_{\theta,\tau_{\theta}^{\ast}\wedge\sigma_{\theta}^{\ast}}
	\left(\mathcal{I}\left(\tau_{\theta}^{\ast},\sigma_{\theta}^{\ast}\right)\right)
	\leq
	\mathcal{E}^{f}_{\theta,\tau_{\theta}^{\ast}\wedge\sigma}
	\left(\mathcal{I}\left(\tau_{\theta}^{\ast},\sigma\right)\right).
\end{equation*}
These assertions follow by adapting the arguments used in the proof of Theorem \ref{Main result}; therefore, the details are omitted.

\section{Perspective}
Several directions for future research naturally arise from the present work. First, it would be interesting to extend the analysis to doubly reflected BSDEs with completely irregular obstacles and to investigate the associated generalized Dynkin games over the class of split stopping times, thereby extending the results obtained in \cite{ArharasOuknine2024,ElmansouriElOtmani2026SplitStopping,GrigorovaImkellerOuknineQuenez2018}. A second direction would be to consider a more general filtration, possibly incorporating a default event, and to study the nonlinear pricing of American game options in richer and more realistic financial market models, in line with the recent developments in \cite{ElmansouriElOtmani2026}. Finally, the numerical approximation of such doubly reflected BSDEs deserves further investigation. In particular, machine-learning-based numerical schemes could be developed by building on the work of Agram et al. \cite{AgramArharasPucciRems2026}, who studied a related problem in the Brownian framework under a deterministic Lipschitz condition on the driver. Another promising direction would be to explore deep reinforcement learning methods for control-oriented applications, particularly for non-Markovian optimal switching problems arising in hybrid renewable energy systems. In this regard, El Hamdi and Hdhiri \cite{ElHamdiHdhiri2027} studied a non-Markovian optimal switching problem for the operation of hybrid renewable energy systems with battery storage, using reflected BSDEs and Dynkin games, and proposed a deep-learning-based numerical approach (see also \cite{HdhiriZaatra2026}). Extending these numerical methods to the present jump setting with stochastic Lipschitz coefficients would constitute a challenging and relevant avenue for future research.

\section*{Conflict of interest} 
The authors declare that they have no competing interests.

\section*{Funding} 
No funding was received for conducting this study.

\end{document}